\documentclass[11pt]{article}
\usepackage[T1]{fontenc}
\usepackage{lmodern}
\usepackage[a4paper,margin=1in]{geometry}
\usepackage{amsmath,amssymb,amsthm,mathtools}
\usepackage{booktabs}
\usepackage[hidelinks]{hyperref}

\newtheorem{theorem}{Theorem}[section]
\newtheorem{proposition}[theorem]{Proposition}
\newtheorem{corollary}[theorem]{Corollary}
\newtheorem{lemma}[theorem]{Lemma}
\newtheorem{remark}[theorem]{Remark}
\newtheorem{definition}[theorem]{Definition}

\newcommand{\KL}{D_{\mathrm{KL}}}
\newcommand{\etaKL}{\eta_{\mathrm{KL}}}
\newcommand{\gapKL}{\gamma_{\mathrm{KL}}}
\newcommand{\rateKL}{\Lambda_{\mathrm{KL}}}
\newcommand{\MI}{I}
\newcommand{\E}{\mathbb{E}}
\newcommand{\TV}{\mathrm{TV}}
\newcommand{\Ent}{\mathrm{Ent}}

\title{Retention Profiles and KL Contraction Bounds in Finite Markov Chains}
\author{
  Saurav Jadhav\thanks{Independent Researcher, Pune, India.
    Email: sauravjadhav698@gmail.com}
}
\date{}

\begin{document}

\maketitle

\begin{abstract}
We study Kullback--Leibler (KL) contraction in finite Markov chains through a
row-wise perspective. Evaluating the SDPI ratio at point masses yields a
state-indexed \emph{retention profile}
$r(x)=\KL(P(x,\cdot)\|\pi)/\log(1/\pi(x))$ and a
\emph{localization ratio} $L(P)=\bar r_\pi/M\in[0,1]$ (with
$M=\max_x r(x)$, $\bar r_\pi=\E_\pi r$) that distinguishes localized from
global contraction obstructions. Our main contributions are
(i) a convexity-gap identity showing that the gap between the row-averaged
divergence and $\KL(\mu P\|\pi)$ equals the mutual information
$I_\mu(X;Y)$, and a derived decomposition of the contraction ratio into
entropy inflation and a mutual-information penalty;
(ii) a Cheeger-type lower bound on $M$, tying the bottleneck geometry of
$P$ directly to the row-retention profile;
(iii) an explicit construction proving that $L(P_n)\to 0$ does not force
$\etaKL(P_n)/M_n\to 1$, identifying cardinality of high-retention states
(not their $\pi$-mass) as the decisive quantity.
Alongside these, we record structural consequences, optimal
Markov/reverse-Markov tail bounds for $r$, a Bhatia--Davis variance bound,
two-sided spectral bounds with an explicit cubic correction, a
KL/Pinsker mixing-time bound, and tensorization for product chains. We
further show that $L(P)$ is structurally decoupled from the spectral gap,
the Cheeger constant, and the mixing time: every vertex-transitive chain
satisfies $L(P)=1$ regardless of its mixing speed, and the empirical rank
correlations between $L(P)$ and these classical invariants on a diverse
but limited test suite are essentially zero. The numerical experiments are
exploratory and not used as evidence for a universal classification
theorem.
\end{abstract}


\section{Introduction and Preliminaries}
\label{sec:prelim}

\subsection{Motivation}

The rate at which a Markov chain forgets its initial state, measured through
Kullback--Leibler divergence to stationarity, is controlled by the
\emph{KL contraction coefficient} $\etaKL(P)$, formally recognized as the
Strong Data Processing Inequality (SDPI) constant for the Markov kernel under
KL divergence~\cite{PolyanskiyWu2017,Raginsky2016,PolyanskiyWu2025}. Computing $\etaKL(P)$ requires
optimizing over the full simplex of initial distributions, making it difficult
to analyze directly. A complementary line of work studies the
\emph{input-dependent} contraction coefficient $\etaKL(\pi,P)$, which fixes
the reference distribution before taking the supremum over the comparison
inputs; this perspective is central to mixing-time analyses for Markov
chains (see~\cite[Def.~33.10 and Ex.~33.8]{PolyanskiyWu2025}, \cite{MakurZheng2015,OrdentlichPolyanskiy2021,GeorgeZhengBansal2024}).
The local KL retention profile $r(x)$ introduced below is precisely the
SDPI ratio defining $\etaKL(P)$ evaluated at the point-mass test measure
$\mu=\delta_x$: since $\KL(\delta_x\|\pi)=\log(1/\pi(x))$ and
$\KL(\delta_x P\|\pi)=\KL(P(x,\cdot)\|\pi)$, one has
$r(x)=\KL(\delta_xP\|\pi)/\KL(\delta_x\|\pi)$. Taking the state index $x$
as the parameter turns this sup-style quantity into a state-indexed profile.
Classical approaches to contraction---spectral gap,
log-Sobolev inequalities~\cite{Gross1975,BobkovTetali2006},
Cheeger constants~\cite{SinclairJerrum1989,LawlerSokal1988} provide
bounds on $\etaKL(P)$
but do not expose whether the worst-case contraction obstruction is
concentrated on a few states or spread across the entire state space.

Thus, the central problem is not merely to bound $\etaKL(P)$, but to
understand where the obstruction to KL contraction is located. Two chains
may have similar global contraction behavior while having very different
spatial profiles of entropy retention. In one chain, the obstruction may be
carried by a few exceptional states, while in another it may be spread
throughout the state space. Classical spectral and isoperimetric quantities
do not directly separate these cases.

This paper develops a \emph{row-based} framework that decomposes the
contraction problem into local contributions from individual transition rows.
We introduce two quantities: the \emph{retention profile} $r(x)$, measuring
the amount of KL divergence each state retains in one step
(cf.\ Cohen et al.~\cite{CohenEtAl1993} and Csisz\'{a}r--Shields~\cite{CsiszarShields2004}
for earlier perspectives on relative-entropy contraction under stochastic
maps), and the
\emph{localization ratio} $L(P)$, measuring the gap between typical and
worst-case retention. These quantities can be computed directly from the
transition kernel without eigenvector decomposition, and they provide
structural information---specifically, the spatial shape of contraction
obstructions that is analytically and empirically distinct from classical
invariants.

Throughout the paper, the term \emph{localization} refers to the spatial
concentration of contraction obstructions (high-$r(x)$ states) within the
state space and should not be confused with eigenvector localization in the
sense of Anderson localization or spectral-graph theory, nor with the
\emph{stochastic localization} of Eldan-type schemes used in sampling and
geometric analysis of measures. (The connection
between the retention profile and eigenvector localization is, however,
explored in Section~\ref{sec:spectral}.)

The central organizing point is that the pair $(\bar r_\pi,L(P))$ separates
the size of the typical row-level obstruction from its spatial concentration.
This separation leads to several spectral, geometric, and mixing-time
consequences. A key result in this direction is the \emph{convexity-gap theorem}
(Theorem~\ref{thm:central-bound}), which identifies the mutual information
$I_\mu(X;Y)$ as the exact convexity gap between the average row divergence and
the divergence after one step.

The main contribution of this paper is the framework itself---the retention
profile $r(x)$ and localization ratio $L(P)$---and the structural
interpretation that they enable. Individual bounds (tail inequalities,
variance bounds, spectral estimates) often follow from standard tools, such
as the Markov and Chebyshev inequalities, the Bhatia--Davis inequality,
and the standard bound of the KL divergence by $\chi^2$ divergence (used in
the proof of Theorem~\ref{thm:spectral-upper}); the value lies in their
unified application to the retention profile and the new perspective this
provides on contraction geometry.

\subsection{Related work}
\label{sec:related}

\paragraph{Strong data processing inequalities and entropy contraction.}
The KL contraction coefficient $\etaKL(P)$ is the SDPI constant of the
Markov kernel under KL divergence. Polyanskiy--Wu~\cite{PolyanskiyWu2017,PolyanskiyWu2025}
and Raginsky~\cite{Raginsky2016} developed systematic frameworks for strong
data processing inequalities in channels and Bayesian networks, including
bounds via $\Phi$-Sobolev inequalities. The \emph{input-dependent}
contraction coefficient $\etaKL(\pi,P)$, in which the reference distribution
is fixed and the supremum is taken over comparison inputs, plays a central
role in mixing-time analysis (\cite[\S33.4]{PolyanskiyWu2025}) and has
recently been studied at the level of general $f$-divergences by
George--Zheng--Bansal~\cite{GeorgeZhengBansal2024},
who show that $\eta_{\chi^2}(P_X,P)$ governs the contraction rate for many
smooth $f$-divergences; Ordentlich--Polyanskiy~\cite{OrdentlichPolyanskiy2021}
establish that the input-\emph{independent} $\etaKL(P)$ is attained at a
binary-input subchannel, and Makur--Zheng~\cite{MakurZheng2015} compare
input-dependent contraction coefficients across $f$-divergences via linear
bounds. Our row-based lower bound (Theorem~\ref{thm:general-row}) probes the
SDPI sup defining $\etaKL(P)$ at the point-mass test measures $\mu=\delta_x$;
the ratio $\KL(\delta_x P\|\pi)/\KL(\delta_x\|\pi)$ is the retention profile
$r(x)=\KL(P(x,\cdot)\|\pi)/\log(1/\pi(x))$. To our knowledge, this
state-indexed evaluation and the derived chain invariants
$\bar r_\pi$, $M$, and $L(P)$ do not appear elsewhere in the SDPI lineage.

\paragraph{Modified log-Sobolev inequalities and entropic curvature.}
The continuous-time modified log-Sobolev inequality, introduced by
Gross~\cite{Gross1975} and extended to discrete settings by
Bobkov--Tetali~\cite{BobkovTetali2006} and
Diaconis--Saloff-Coste~\cite{DiaconisSaloffCoste1996}, controls entropy
decay along Markov semigroups. Remark~\ref{rmk:mlsi-equivalence} restates
the folklore reformulation of $\etaKL(P)$ as the discrete-time full-step
entropy contraction constant, clarifying the relationship between these
perspectives (see also~\cite[Eq.~(20) and \S2]{CaputoChenGuPolyanskiy2024}).
Miclo~\cite{Miclo2015} demonstrated that discrete-time and continuous-time
entropy contraction can behave differently, motivating our focus on the
exact one-step setting. A parallel curvature-based line of work descended
from Bakry--\'Emery and Ollivier and developed for discrete chains by
Caputo--M\"unch--Salez~\cite{CaputoMunchSalez2024},
Caputo--Salez~\cite{CaputoSalez2024},
Salez~\cite{Salez2026},
M\"unch--Salez~\cite{MunchSalez2022},
Kamtue--Liu--M\"unch--Peyerimhoff~\cite{KamtueLiuMunchPeyerimhoff2023},
Pedrotti~\cite{Pedrotti2023},
and Bristiel--Caputo~\cite{BristielCaputo2021}, derives MLSI and entropy
decay from non-negative discrete Ricci curvature; this is complementary to
our purely informational, row-level approach. Cutoff phenomena under
non-negative curvature were established by Salez~\cite{Salez2021Cutoff},
and isoperimetric routes to total-variation decay were developed by
Hutchcroft--Lopez~\cite{HutchcroftLopez2024}.

\paragraph{Cheeger bounds and spectral methods.}
The classical Cheeger inequality~\cite{SinclairJerrum1989,LawlerSokal1988}
relates the spectral gap to an isoperimetric constant. Our Cheeger-type
bound (Theorem~\ref{thm:cheeger-M}) operates at the level of individual
transition rows rather than the global spectrum, connecting the bottleneck
geometry directly to the retention profile.
Chung~\cite{Chung1997} provides foundational background on spectral graph
theory, which informs our eigenvector-localization analysis in
Section~\ref{sec:spectral}.

\paragraph{Relative entropy under stochastic maps.}
Cohen et al.~\cite{CohenEtAl1993} studied contraction of relative entropy
under stochastic matrices, and
Csisz\'{a}r--Shields~\cite{CsiszarShields2004} surveyed the broader
information-theoretic landscape. Our retention profile $r(x)$ refines
these perspectives by localizing the contraction analysis to individual rows
and measuring their relative contribution under the stationary law.
The hypercontractivity approach of
Ahlswede--G\'{a}cs~\cite{AhlswedeGacs1976} and the maximal-correlation
framework of Anantharam et al.~\cite{AnantharamEtAl2013} provide related
but distinct routes to data processing inequalities.

\paragraph{Mixing times.}
Levin--Peres~\cite{LevinPeres2017} and
Montenegro--Tetali~\cite{MontenegroTetali2006} provide comprehensive
treatments of mixing-time bounds via spectral and functional-inequality
methods. Our mixing-time bound (Theorem~\ref{thm:mixing-upper}) is not
competitive with the sharpest results in these references; its purpose is to
connect the row-based quantities to the mixing time, completing the
structural picture.

\subsection{Organization}

Section~\ref{sec:row-bound} establishes the row-based lower bound
and its connection to discrete-time entropy contraction.
Section~\ref{sec:retention} introduces the retention profile and
localization ratio with optimal tail bounds.
Section~\ref{sec:variance} presents the variance--retention bound.
Section~\ref{sec:spectral} develops two-sided spectral connections.
Section~\ref{sec:cheeger} proves the Cheeger-type lower bound.
Section~\ref{sec:convexity-gap} presents the convexity gap theorem.
Section~\ref{sec:mixing} derives mixing-time consequences.
Section~\ref{sec:examples} presents examples, and
Section~\ref{sec:discussion} presents the conclusion and discusses open
questions.
Appendix~\ref{app:counterexample} provides a formal construction proving the
limits of $L(P)$ as a diagnostic for $\etaKL(P)/M$, and
Appendix~\ref{sec:tensor} records tensorization properties for product
chains.

\subsection{Setup and notation}

Let $P$ be a Markov transition kernel on a finite state space $S$ with
$|S|=n\ge2$. Assume that $P$ admits a stationary distribution $\pi$ with full
support, i.e.,
\[
  \pi(x)>0, \qquad \forall x\in S.
\]
Then $0<\pi(x)<1$ for every $x\in S$. We write
\[
  \pi_{\min}:=\min_{x\in S}\pi(x),\qquad
  \pi_{\max}:=\max_{x\in S}\pi(x).
\]

For probability measures $\mu,\nu$ on $S$, define the Kullback--Leibler
divergence (see, e.g.,~\cite{CoverThomas2006}; see also~\cite{Amari2016} for the
information-geometric perspective) by
\[
  \KL(\mu\|\nu)
  =
  \sum_{x\in S}\mu(x)\log\frac{\mu(x)}{\nu(x)}.
\]
All logarithms are natural logarithms throughout.

For any initial distribution $\mu\neq\pi$, we refer to the quotient
$\KL(\mu P\|\pi)/\KL(\mu\|\pi)$ as the \emph{KL contraction ratio}
of~$\mu$. Define the \emph{KL contraction coefficient} of $P$ as the
supremum of this ratio:
\[
  \etaKL(P)
  :=
  \sup_{\mu\neq\pi}
  \frac{\KL(\mu P\|\pi)}{\KL(\mu\|\pi)}.
\]
Define the \emph{KL contraction gap} by
$\gapKL(P):=1-\etaKL(P)$.
When $\etaKL(P)>0$, define the \emph{exponential decay rate} by
$\rateKL(P):=-\log\etaKL(P)$.
Since $\pi P=\pi$, the data-processing inequality gives
\[
  \KL(\mu P\|\pi)=\KL(\mu P\|\pi P)\le \KL(\mu\|\pi),
\]
so $0\le \etaKL(P)\le 1$.

For each $x\in S$, the row-to-stationary divergence is written as
\[
  \KL\bigl(P(x,\cdot)\,\|\,\pi\bigr)
  =
  \sum_{y\in S}P(x,y)\log\frac{P(x,y)}{\pi(y)}.
\]

\begin{definition}[Entropy functional]
The \emph{entropy functional} relative to~$\pi$ is
\[
  \Ent_\pi(f)
  := \sum_{x\in S} f(x)\log f(x)\,\pi(x)
     - \Bigl(\sum_{x\in S} f(x)\pi(x)\Bigr)
       \log\Bigl(\sum_{x\in S} f(x)\pi(x)\Bigr)
\]
for non-negative functions $f\ge 0$ on $S$.
When $\mu$ has density $f=d\mu/d\pi$ with respect to $\pi$, one has
$\Ent_\pi(f) = \KL(\mu\|\pi)$.
\end{definition}

When $P$ is \emph{reversible} with respect to $\pi$ (i.e.,
$\pi(x)P(x,y)=\pi(y)P(y,x)$ for all $x,y$), we write
$1=\lambda_1\ge\lambda_2\ge\cdots\ge\lambda_n\ge -1$ for the eigenvalues of
$P$ in $L^2(\pi)$, with the corresponding orthonormal eigenfunctions
$\phi_1\equiv 1,\phi_2,\ldots,\phi_n$.

Define the \emph{absolute spectral gap} by
\[
  \gamma(P):=1-\max_{i\ge 2}|\lambda_i|.
\]
Define the \emph{Cheeger constant} of $P$
(see~\cite{SinclairJerrum1989,LawlerSokal1988}) by
\[
  h_P:=\min_{\substack{A\subset S\\0<\pi(A)\le 1/2}}
  \frac{\sum_{x\in A,\,y\notin A}\pi(x)P(x,y)}{\pi(A)}.
\]
Define the \emph{total-variation mixing time} by
\[
  t_{\mathrm{mix}}(\varepsilon)
  :=\min\bigl\{t\ge 0:\max_{x\in S}\|P^t(x,\cdot)-\pi\|_{\TV}\le\varepsilon\bigr\},
\]
where $\|\mu-\nu\|_{\TV}=\frac{1}{2}\sum_{x\in S}|\mu(x)-\nu(x)|$.

It is crucial to distinguish the discrete-time contraction of the kernel $P$
from the continuous-time mixing governed by semigroup $e^{t(P-I)}$.
Recent work has demonstrated that the discrete-time full-step contraction
(characterized by the SDPI constant) and the continuous-time entropy
contraction (characterized by the classical modified log-Sobolev inequality) can behave
very differently; see~\cite{Miclo2015} for a detailed comparison in the
context of hyperboundedness (see also Ahlswede--G\'{a}cs~\cite{AhlswedeGacs1976}
and Anantharam et al.~\cite{AnantharamEtAl2013} for hypercontractivity and
maximal-correlation approaches to data processing inequalities).
This motivated us to focus on the exact discrete-time step.

\section{Row-Based Lower Bound and One-Step Entropy Contraction}
\label{sec:row-bound}

The starting observation is that evaluating the contraction coefficient at
point masses $\mu=\delta_x$ yields a lower bound on $\etaKL(P)$ in terms of
individual transition rows. This is valid because $\etaKL(P)$ is defined as
a supremum over \emph{all} $\mu\neq\pi$; in particular, since $\pi$ has full support ($\pi(x)>0$ for all $x\in S$), every point mass
$\delta_x$ is an admissible test distribution.

\begin{theorem}[Row-based lower bound]
\label{thm:general-row}
Let $P$ be a Markov kernel on a finite state space $S$ with $|S|\ge2$ and
stationary distribution $\pi$ of full support. Then
\[
  \etaKL(P)
  \;\ge\;
  \max_{x\in S}
  \frac{\KL\bigl(P(x,\cdot)\,\|\,\pi\bigr)}{\log\!\bigl(1/\pi(x)\bigr)}.
\]
Equivalently,
$\gapKL(P)\le 1-
\max_{x\in S}
\frac{\KL\bigl(P(x,\cdot)\,\|\,\pi\bigr)}{\log\!\bigl(1/\pi(x)\bigr)}$.
\end{theorem}

\begin{proof}
Fix $x\in S$ and choose $\mu=\delta_x$.
Because $\pi$ has full support, $\KL(\delta_x\|\pi)=\log(1/\pi(x))$.
Moreover,
\[
  \KL(\delta_x P\|\pi)
  =
  \sum_{y\in S}P(x,y)\log\frac{P(x,y)}{\pi(y)}
  =
  \KL\bigl(P(x,\cdot)\,\|\,\pi\bigr).
\]
Therefore,
\[
  \frac{\KL(\delta_xP\|\pi)}{\KL(\delta_x\|\pi)}
  =
  \frac{\KL(P(x,\cdot)\|\pi)}{\log(1/\pi(x))}.
\]
Since $\etaKL(P)$ is the supremum over all $\mu\neq\pi$, we have
$\etaKL(P)\ge \KL(\delta_xP\|\pi)/\KL(\delta_x\|\pi)$ for each $x\in S$.
Taking the maximum over $x$ proves the first claim. Subtracting from $1$
yields the gap form.
\end{proof}

\begin{remark}[Discrete-time entropy contraction equivalence; folklore]
\label{rmk:mlsi-equivalence}
Let $P$ be a finite, \emph{reversible} (with respect to~$\pi$),
discrete-time Markov kernel. Define the \emph{discrete-time full-step
entropy contraction gap} by
\[
  \alpha_1(P) := \inf_{f > 0,\; f \not\equiv \text{const}}
  \frac{\Ent_\pi(f) - \Ent_\pi(Pf)}{\Ent_\pi(f)}.
\]
Then $\etaKL(P) = 1 - \alpha_1(P)$. This identity is the standard
unrolling of the sup-form of $\etaKL$ in densities (cf.\
\cite[Eq.~(20)]{CaputoChenGuPolyanskiy2024} and
\cite{PolyanskiyWu2017,PolyanskiyWu2025}); we record the short proof below
for completeness and to fix the notation used in subsequent sections.
\end{remark}

\begin{proof}[Proof of Remark~\ref{rmk:mlsi-equivalence}]
The ratio $[\Ent_\pi(f)-\Ent_\pi(Pf)]/\Ent_\pi(f)$ is invariant under
positive scaling $f\mapsto cf$ (since $\Ent_\pi(cf)=c\,\Ent_\pi(f)$ and
$P$ is linear), so the infimum over all $f>0$ with $f\not\equiv\text{const}$
equals the infimum over the subset of $\pi$-densities (i.e.,
$f>0$ with $\E_\pi[f]=1$). Working with densities, reversibility of $P$
is used to show that $Pf$ is the $\pi$-density of
$\mu P$ when $f = d\mu/d\pi$. Specifically, for $f = d\mu/d\pi > 0$, by
reversibility $\pi(z)P(z,y)=\pi(y)P(y,z)$, hence
\[
  (Pf)(y)
  = \sum_z P(y,z) f(z)
  = \sum_z \frac{\pi(z)P(z,y)}{\pi(y)}\,\frac{\mu(z)}{\pi(z)}
  = \frac{(\mu P)(y)}{\pi(y)},
\]
which confirms that $Pf$ is the $\pi$-density of $\mu P$. Consequently
$\Ent_\pi(f) = \KL(\mu\|\pi)$ and $\Ent_\pi(Pf) = \KL(\mu P\|\pi)$, giving
\[
  \frac{\Ent_\pi(f) - \Ent_\pi(Pf)}{\Ent_\pi(f)}
  = 1 - \frac{\KL(\mu P\|\pi)}{\KL(\mu\|\pi)}.
\]
Thus, the variational formula above holds for all full-support distributions
$\mu\neq\pi$. This is enough to recover the same supremum as in the definition
of $\etaKL(P)$: if $\mu$ is any distribution, set
$\mu_\varepsilon=(1-\varepsilon)\mu+\varepsilon\pi$. Then
$\mu_\varepsilon$ has full support and, since $\pi$ has full support,
\[
  \KL(\mu_\varepsilon\|\pi)\to\KL(\mu\|\pi),
  \qquad
  \KL(\mu_\varepsilon P\|\pi)\to\KL(\mu P\|\pi)
\]
as $\varepsilon\downarrow0$. Therefore the supremum over full-support
$\mu\neq\pi$ equals the supremum over all $\mu\neq\pi$, and hence
\[
  \alpha_1(P)
  = \inf_{\mu\neq\pi}\Bigl(1 - \frac{\KL(\mu P\|\pi)}{\KL(\mu\|\pi)}\Bigr)
  = 1 - \sup_{\mu\neq\pi}\frac{\KL(\mu P\|\pi)}{\KL(\mu\|\pi)}
  = 1 - \etaKL(P). \qedhere
\]
\end{proof}

\begin{remark}[Terminology and prior occurrences]
\label{rmk:mlsi-term}
The constant $\alpha_1(P)$ above is sometimes written $\delta(\pi,P)$ in
the SDPI literature (cf.\ \cite[Eq.~(20)]{CaputoChenGuPolyanskiy2024}); it
is the discrete-time full-step entropy contraction gap, and is distinct
from the continuous-time modified log-Sobolev constant studied in, e.g.,
Gross~\cite{Gross1975}, Bobkov--Tetali~\cite{BobkovTetali2006}, and
Diaconis--Saloff-Coste~\cite{DiaconisSaloffCoste1996}. The reformulation
$\etaKL(P)=1-\alpha_1(P)$ is folklore; we use it here as a
functional-inequality reformulation of the SDPI constant of
Polyanskiy--Wu~\cite{PolyanskiyWu2017,PolyanskiyWu2025} and
Raginsky~\cite{Raginsky2016}.
\end{remark}

\begin{remark}[Point-mass entropy contraction ratio]
\label{rmk:mlsi}
For the point-mass test function $f_x(y) = \delta_{xy}/\pi(x)$,
we have $\Ent_\pi(f_x) = \log(1/\pi(x))$. A direct computation
(see Appendix~\ref{app:mlsi-detail}) gives
$\Ent_\pi(Pf_x) = \KL(P(x,\cdot)\|\pi)$ when $P$ is reversible. Thus
\[
  r(x)
  = \frac{\KL(P(x,\cdot)\|\pi)}{\log(1/\pi(x))}
  = \frac{\Ent_\pi(Pf_x)}{\Ent_\pi(f_x)},
\]
and $M = \max_x r(x)$ is the \emph{point-mass entropy contraction ratio}. The
identity $\etaKL(P)=M$ holds if and only if the row-based lower bound is
tight, equivalently, if the discrete-time entropy contraction gap is attained
or approached by boundary point-mass tests.
\end{remark}

\begin{corollary}[Row-entropy bound for uniform $\pi$]
\label{cor:uniform}
Assume $\pi$ is uniform on $S$, that is, $\pi(x)=1/n$.
Define
\[
  H(P(x,\cdot))=-\sum_{y\in S}P(x,y)\log P(x,y),
  \qquad
  H_{\min}(P):=\min_{x\in S}H(P(x,\cdot)).
\]
Then
\[
  \etaKL(P)\ge 1-\frac{H_{\min}(P)}{\log n},
  \qquad
  \gapKL(P)\le \frac{H_{\min}(P)}{\log n}.
\]
\end{corollary}

\begin{proof}
If $\pi$ is uniform, then $\log(1/\pi(x))=\log n$ for all $x\in S$. Also
\[
  \KL(P(x,\cdot)\|\pi)=\sum_{y\in S}P(x,y)\log\frac{P(x,y)}{1/n}
  =\log n - H(P(x,\cdot)).
\]
Substituting into Theorem~\ref{thm:general-row}:
\[
  \etaKL(P)\ge \max_x \frac{\log n-H(P(x,\cdot))}{\log n}
  =1-\frac{H_{\min}(P)}{\log n}.
\]
Subtracting from $1$ gives $\gapKL(P)\le H_{\min}(P)/\log n$.
\end{proof}

\section{Local Retention Profile and Localization}
\label{sec:retention}

\begin{definition}[Local KL retention and typical retention]
\label{def:r}
Define the \emph{local KL retention ratio}
\[
  r(x):=
  \frac{\KL(P(x,\cdot)\|\pi)}{\log(1/\pi(x))}.
\]
Define the \emph{typical retention} (stationary average)
\[
  \bar r_\pi:=\E_{X\sim\pi}[r(X)]=\sum_{x\in S}\pi(x)\,r(x).
\]
By data processing applied to $\delta_x$, $0\le r(x)\le1$ for every state.
\end{definition}

\begin{proposition}[Hierarchy: typical vs worst local vs worst global]
\label{prop:hierarchy}
For any $P$ with stationary distribution $\pi$,
\[
  \bar r_\pi \;\le\; \max_{x\in S} r(x) \;\le\; \etaKL(P).
\]
\end{proposition}

\begin{proof}
The first inequality holds because $\bar r_\pi$ is a convex combination of
the values $\{r(x)\}_{x\in S}$ with weights $\pi(x)\ge 0$ summing to~$1$.
The second inequality is given by Theorem~\ref{thm:general-row}.
\end{proof}

\begin{definition}[Localization ratio]
\label{def:rho}
Define
\[
  M:=\max_{x\in S} r(x),
  \qquad
  L(P):=\frac{\bar r_\pi}{M},
\]
with the convention $L(P)=1$ when $\bar r_\pi=M=0$.
\end{definition}

\begin{remark}[Interpretation]
\label{rmk:interpretation}
The quantity $L(P)\in[0,1]$ measures how representative the worst local
obstruction is under $\pi$:
\begin{itemize}
\item $L(P)\approx 1$ indicates a \emph{global obstruction}---typical
  states under~$\pi$ already exhibit near-worst local retention.
\item $L(P)$ well below~$1$ indicates a \emph{localized
  obstruction}---the worst local behavior is confined to a set of negligible
  $\pi$-mass.
\end{itemize}
(The convention $L(P)=1$ when $M=0$ corresponds to the trivial chain
$P(x,\cdot)=\pi$ for all~$x$, which has no contraction obstruction.)

The localization ratio is intended as a \emph{diagnostic invariant} that
captures the spatial shape of contraction obstructions. It does not
replace the spectral gap or Cheeger constant, but quantifies an aspect of
the chain geometry not captured by those classical quantities.

More generally, when $M>0$, we can define
$L_\nu(P) := \E_\nu[r(X)]/M$ for any distribution $\nu$ on~$S$; the stationary specialization
$L(P) = L_\pi(P)$ is canonical because it yields a chain invariant
independent of the initial conditions (see also Remark~\ref{rmk:transient}).
\end{remark}

\subsection{Optimal tail bounds}

\begin{theorem}[Optimal two-sided tail bounds under an arbitrary law]
\label{thm:ratio-characterization-general}
\leavevmode\par\noindent
Assume $M:=\max_{x\in S} r(x)>0$,
$Y:=r(X)/M\in[0,1]$,
and let $\nu$ be any probability distribution on $S$.
If $X\sim \nu$, define
$m_\nu := \E_\nu[Y] = \sum_{x\in S}\nu(x)\,r(x)/M$.
Then, for every $\tau\in(0,1)$,
\[
  \max\!\left\{0,\frac{m_\nu-\tau}{1-\tau}\right\}
  \le
  \nu\!\left(\{x\in S:\, r(x)\ge \tau M\}\right)
  \le
  \min\!\left\{1,\frac{m_\nu}{\tau}\right\},
\]
and equivalently,
\[
  \max\!\left\{0,1-\frac{m_\nu}{\tau}\right\}
  \le
  \nu\!\left(\{x\in S:\, r(x)\le \tau M\}\right)
  \le
  \min\!\left\{1,\frac{1-m_\nu}{1-\tau}\right\}.
\]
These bounds are sharp: given only $m_\nu$ and $M$, no stronger universal
tail bounds can be obtained. Depending on weak versus strict threshold
conventions, the extremal constants are either attained or approached by
two-point laws supported at endpoint/threshold values.
\end{theorem}

\begin{proof}
Define $Y(x):=r(x)/M \in [0,1]$ for every $x\in S$.

\medskip
\noindent\textbf{Upper bound on the upper tail.}
Markov's inequality gives
$\nu(Y\ge \tau)\le m_\nu/\tau$.
Since probabilities are at most $1$,
$\nu(Y\ge \tau)\le \min\{1,\,m_\nu/\tau\}$.

\medskip
\noindent\textbf{Lower bound on the upper tail.}
Decomposing the expectation:
\[
  m_\nu = \mathbb{E}_\nu[Y]
  = \mathbb{E}_\nu[Y\mathbf{1}_{\{Y\le \tau\}}]
  + \mathbb{E}_\nu[Y\mathbf{1}_{\{Y>\tau\}}]
  \le \tau\,\nu(Y\le \tau) + \nu(Y>\tau).
\]
Since $\nu(Y\le \tau) = 1 - \nu(Y > \tau)$, we have:
\[
  m_\nu \le \tau(1-\nu(Y>\tau)) + \nu(Y>\tau)
  = \tau + (1-\tau)\nu(Y>\tau).
\]
Rearranging terms implies $\nu(Y>\tau) \ge (m_\nu - \tau)/(1-\tau)$. Since $\nu(Y\ge\tau)\ge\nu(Y>\tau)$ and probabilities are non-negative,
\[
  \nu(Y\ge\tau) \ge \max\left\{0, \frac{m_\nu-\tau}{1-\tau}\right\}.
\]

\medskip
\noindent\textbf{Complementary tail.}
The lower-tail bounds follow from the identity
$\nu(Y\le\tau) = 1 - \nu(Y>\tau)$ applied to the above inequalities.

\medskip
\noindent\textbf{Optimality.}
The Markov upper bound is attained when $m_\nu\le\tau$, by putting mass
$m_\nu/\tau$ at $\tau$ and the remaining mass at~$0$; when $m_\nu>\tau$ it
is trivially equal to~$1$. The lower bound on the upper tail is approached,
when $m_\nu>\tau$, by putting mass $(m_\nu-\tau)/(1-\tau)$ at~$1$ and the
remaining mass just below~$\tau$; when $m_\nu\le\tau$, it is attained by
laws supported below~$\tau$. Applying the same constructions to $1-Y$ gives
the complementary lower-tail statements. Hence, the displayed constants cannot
be improved using only the mean and the support constraint $Y\in[0,1]$.
\end{proof}

\begin{corollary}[Stationary specialization]
\label{cor:ratio-characterization-stationary}
If $M>0$ and $\nu=\pi$, then $m_\pi=L(P)$, and for every $\tau\in(0,1)$,
\[
  \max\!\left\{0,\frac{L(P)-\tau}{1-\tau}\right\}
  \le
  \pi\!\left(\{x\in S:\, r(x)\ge \tau M\}\right)
  \le
  \min\!\left\{1,\frac{L(P)}{\tau}\right\}.
\]
\end{corollary}

\begin{proof}
Immediate from Theorem~\ref{thm:ratio-characterization-general} with
$\nu=\pi$ and $m_\pi = \bar r_\pi / M = L(P)$.
\end{proof}

\begin{remark}[Transient version]
\label{rmk:transient}
If $M>0$ and $X_t\sim \mu P^t$ for some initial distribution $\mu$, the same bounds
hold with $m_\nu$ replaced by
$m_t:=\sum_{x}(\mu P^t)(x)\,r(x)/M$.
If the chain is ergodic, then $\mu P^t \to \pi$ and hence
$m_t \to L(P)$ as $t\to\infty$.
\end{remark}

\begin{remark}[Nature of the bounds]
The upper tail bound is Markov's inequality; the lower tail bound is its
complementary averaging counterpart. The contribution lies in the
application: these bounds provide tight control of the $\pi$-mass of
bottleneck states from only two summary statistics ($L(P)$ and $M$),
without requiring spectral decomposition or explicit knowledge of $r(x)$
at each state.
\end{remark}

\subsection{Uniform stationary distribution and mutual information}

When $\pi$ is uniform, $\log(1/\pi(x))=\log n$ is constant, so
\[
  \bar r_\pi
  =
  \frac{1}{\log n}\sum_{x\in S}\pi(x)\KL(P(x,\cdot)\|\pi).
\]

\begin{definition}[Mutual information for the stationary one-step pair]
\label{def:mi}
Let $(X_0,X_1)$ be defined by
$X_0\sim\pi$ and $X_1\mid X_0=x\sim P(x,\cdot)$.
The mutual information is
\[
  \MI(X_0;X_1)
  :=
  \sum_{x,y\in S}\pi(x)P(x,y)\log
  \frac{P(x,y)}{\pi(y)}.
\]
\end{definition}

\begin{proposition}[Uniform $\pi$: $\bar r_\pi$ equals normalized mutual information]
\label{prop:uniform-mi}
Assume $\pi$ is uniform on $S$ with $|S|=n$. Then
$\bar r_\pi=\MI(X_0;X_1)/\log n$.
\end{proposition}

\begin{proof}
By definition,
\[
  \MI(X_0;X_1)
  =\sum_{x,y}\pi(x)P(x,y)\log\frac{P(x,y)}{\pi(y)}
  =\sum_x \pi(x)\KL(P(x,\cdot)\|\pi).
\]
If $\pi$ is uniform:
\[
  \bar r_\pi=\sum_x \pi(x)\frac{\KL(P(x,\cdot)\|\pi)}{\log n}
  =\frac{\MI(X_0;X_1)}{\log n}.
\]
\end{proof}

\begin{remark}
Proposition~\ref{prop:uniform-mi} is the equilibrium identity.
It does not yield bounds on the worst-case coefficient $\etaKL(P)$, which
involves a supremum over all initial distributions.
\end{remark}

\section{Variance--Retention Bound}
\label{sec:variance}

The localization ratio also governs the dispersion of the retention profile
under~$\pi$.

\begin{proposition}[Sharp variance bound]
\label{prop:variance}
For any $P$ with stationary distribution $\pi$,
\[
  \mathrm{Var}_\pi(r)
  :=
  \sum_{x\in S}\pi(x)\bigl(r(x)-\bar r_\pi\bigr)^2
  \;\le\;
  \bar r_\pi\,M\,(1-L(P)).
\]
This bound is sharp: equality holds for any chain in which $r(x)\in\{0,M\}$
for $\pi$-a.e.\ $x$.
\end{proposition}

\begin{proof}
Because $0\le r(x)\le M$ for all $x$, pointwise $r(x)^2\le M\cdot r(x)$.
Averaging over~$\pi$:
$\E_\pi[r^2]\le M\,\bar r_\pi$.
Therefore,
\[
  \mathrm{Var}_\pi(r)
  =\E_\pi[r^2]-\bar r_\pi^2
  \le M\,\bar r_\pi-\bar r_\pi^2
  =\bar r_\pi\,M\,(1-L(P)).
\]
For sharpness, suppose $r(x)\in\{0,M\}$ with $\pi(\{x:r(x)=M\})=p$.
Then $\bar r_\pi=pM$, $L(P)=p$, and
$\mathrm{Var}_\pi(r)=p(1-p)M^2=\bar r_\pi\cdot M\cdot(1-L(P))$,
achieving equality.
\end{proof}

\begin{remark}[Relation to the Bhatia--Davis inequality]
Proposition~\ref{prop:variance} is an instance of the
Bhatia--Davis inequality~\cite{BhatiaDavis2000}: for a random variable
$X\in[a,b]$, $\mathrm{Var}(X)\le (b-\E[X])(\E[X]-a)$. Applied with
$a=0$, $b=M$, $\E_\pi[r]=\bar r_\pi$, it yields the stated bound.
The interpretation is that $L(P)\approx 1$ forces
$\mathrm{Var}_\pi(r)\approx 0$, i.e., a global obstruction implies
approximate spatial homogeneity of the entropy retention profile.
\end{remark}

\begin{corollary}[Concentration of the retention profile]
\label{cor:concentration}
For every $\varepsilon>0$,
\[
  \pi\bigl(\{x\in S:|r(x)-\bar r_\pi|>\varepsilon\}\bigr)
  \;\le\;
  \frac{\bar r_\pi\,M\,(1-L(P))}{\varepsilon^2}.
\]
In particular, as $L(P)\to 1$, the profile $r(\cdot)$ converges to the
constant $\bar r_\pi$ in $L^2(\pi)$.
\end{corollary}

\begin{proof}
Apply Chebyshev's inequality to the bound in Proposition~\ref{prop:variance}.
\end{proof}

\section{Spectral Connections}
\label{sec:spectral}

Throughout this section, $P$ is \emph{reversible} with respect to~$\pi$.

\subsection{Near-stationary lower bound}

The row-based lower bound uses point masses. A second, independent source of
large KL contraction comes from distributions infinitesimally close to
stationarity. This observation is useful both analytically and as a numerical
calibration, because many examples in the empirical section are controlled by
this local quadratic regime rather than by boundary point masses.

\begin{proposition}[Local quadratic lower bound]
\label{prop:local-l2-lower}
Let $P$ be reversible with respect to $\pi$. Then
\[
  \etaKL(P)\ge \max_{i\ge 2}|\lambda_i|^2.
\]
More precisely, if $0\neq f\in L^2(\pi)$ satisfies $\E_\pi f=0$, then for all
sufficiently small $\varepsilon$ such that $1+\varepsilon f>0$, the distribution
$d\mu_\varepsilon/d\pi=1+\varepsilon f$ satisfies
\[
  \lim_{\varepsilon\to0}
  \frac{\KL(\mu_\varepsilon P\|\pi)}{\KL(\mu_\varepsilon\|\pi)}
  =
  \frac{\|Pf\|_{L^2(\pi)}^2}{\|f\|_{L^2(\pi)}^2}.
\]
\end{proposition}

\begin{proof}
For $\E_\pi f=0$ and $d\mu_\varepsilon/d\pi=1+\varepsilon f$, expand
$(1+u)\log(1+u)=u+\tfrac{u^2}{2}+O(u^3)$ at $u=\varepsilon f(x)$ and
average with respect to $\pi$. Using $\sum_x\pi(x)f(x)=0$ and the
boundedness of $f$ on the finite state space $S$,
\begin{align*}
  \KL(\mu_\varepsilon\|\pi)
  &= \sum_{x\in S}\pi(x)\,(1+\varepsilon f(x))\log(1+\varepsilon f(x)) \\
  &= \sum_{x\in S}\pi(x)\Bigl[\varepsilon f(x) + \tfrac{\varepsilon^2}{2}f(x)^2
     + O(\varepsilon^3 f(x)^3)\Bigr] \\
  &= \frac{\varepsilon^2}{2}\|f\|_{L^2(\pi)}^2 + o(\varepsilon^2).
\end{align*}
By reversibility,
\[
  \sum_x \pi(x)f(x)P(x,y)
  = \pi(y)\sum_x P(y,x)f(x)
  = \pi(y)(Pf)(y),
\]
so $(\mu_\varepsilon P)(y)=\pi(y)(1+\varepsilon (Pf)(y))$ and the density
of $\mu_\varepsilon P$ with respect to $\pi$ is $1+\varepsilon Pf$.
Moreover $\sum_x\pi(x)(Pf)(x)=\sum_x\pi(x)f(x)=0$ (stationarity), so the
same expansion gives
\[
  \KL(\mu_\varepsilon P\|\pi)
  =\frac{\varepsilon^2}{2}\|Pf\|_{L^2(\pi)}^2+o(\varepsilon^2).
\]
Dividing yields the displayed limit. Taking $f$ to be an eigenfunction
corresponding to a nontrivial eigenvalue of the largest absolute value gives
$\etaKL(P)\ge\max_{i\ge2}|\lambda_i|^2$. The bound
$\etaKL(P)\ge\max_{i\ge2}\lambda_i^2$ is classical and corresponds to the
second-order ($\chi^2$) lower bound on KL contraction; see
Polyanskiy--Wu~\cite[Ch.~33]{PolyanskiyWu2025}.
\end{proof}

\subsection{Spectral upper bound}

\begin{lemma}[$\chi^2$--spectral identity]
\label{lem:chi2}
For each $x\in S$,
\[
  \chi^2(P(x,\cdot)\|\pi)
  :=\sum_{y\in S}\frac{(P(x,y)-\pi(y))^2}{\pi(y)}
  =\sum_{i=2}^n\lambda_i^2\,\phi_i(x)^2.
\]
Consequently,
$\sum_{x\in S}\pi(x)\,\chi^2(P(x,\cdot)\|\pi)=\sum_{i=2}^n\lambda_i^2$.
\end{lemma}

\begin{proof}
The spectral decomposition of the kernel gives
$P(x,y)/\pi(y)=\sum_{i=1}^n\lambda_i\phi_i(x)\phi_i(y)$. Hence
\begin{align*}
  \chi^2(P(x,\cdot)\|\pi)
  &=\sum_{y}\pi(y)\Bigl(\frac{P(x,y)}{\pi(y)}-1\Bigr)^{\!2} \\
  &=\sum_{y}\pi(y)\Bigl(\sum_{i\ge 2}\lambda_i\phi_i(x)\phi_i(y)\Bigr)^{\!2} \\
  &=\sum_{i\ge 2}\lambda_i^2\,\phi_i(x)^2,
\end{align*}
because of the orthonormality of $\{\phi_i\}$ in $L^2(\pi)$.
Averaging over $\pi$ and using $\sum_x\pi(x)\phi_i(x)^2=1$ gives the global
identity.
\end{proof}

\begin{theorem}[Spectral upper bound on typical retention]
\label{thm:spectral-upper}
For a reversible chain with stationary distribution $\pi$,
\[
  \MI(X_0;X_1)
  =\sum_{x\in S}\pi(x)\,\KL(P(x,\cdot)\|\pi)
  \;\le\;\sum_{i=2}^n\lambda_i^2.
\]
Consequently,
\[
  \bar r_\pi\;\le\;
  \frac{\sum_{i=2}^n\lambda_i^2}{\log(1/\pi_{\max})},
\]
where $\pi_{\max}=\max_x\pi(x)$.
If $\pi$ is uniform, then
\[
  \bar r_\pi
  \;\le\;
  \frac{\sum_{i=2}^n\lambda_i^2}{\log n}
  \;\le\;
  \frac{(n-1)(1-\gamma)^2}{\log n}.
\]
\end{theorem}

\begin{proof}
The standard inequality $\KL(p\|q)\le\chi^2(p\|q)$
(following from $\log t\le t-1$) applied row-wise and averaged over $\pi$
gives
\[
  \MI(X_0;X_1)
  \le\sum_x\pi(x)\chi^2(P(x,\cdot)\|\pi)
  =\sum_{i=2}^n\lambda_i^2,
\]
by Lemma~\ref{lem:chi2}. Since
$\log(1/\pi(x))\ge\log(1/\pi_{\max})$ for every $x$,
\[
  \bar r_\pi
  =\sum_x\pi(x)\frac{\KL(P(x,\cdot)\|\pi)}{\log(1/\pi(x))}
  \le\frac{\sum_x\pi(x)\KL(P(x,\cdot)\|\pi)}{\log(1/\pi_{\max})}
  \le\frac{\sum_{i=2}^n\lambda_i^2}{\log(1/\pi_{\max})}.
\]
The uniform case follows from $\pi_{\max}=1/n$ and
$\sum_{i\ge 2}\lambda_i^2\le(n-1)(1-\gamma)^2$.
\end{proof}

\subsection{Conditional spectral lower bound via second-order expansion}

The upper bound of Theorem~\ref{thm:spectral-upper} arises from $\KL \le \chi^2$.
We now establish a conditional converse, valid when transition rows are
sufficiently close to~$\pi$ (bounded row regime), by quantifying the gap via
Taylor expansion with an explicit third-moment correction. The lower bound
is most useful for chains with bounded row entries $P(x,y)/\pi(y)$; for
sparse chains or those with near-absorbing states, the correction term may
dominate (see Corollary~\ref{cor:spectral-sandwich} and Remark~\ref{rmk:spectral-caveat} for the precise conditions).

\begin{theorem}[Spectral lower bound on mutual information]
\label{thm:spectral-lower}
For a reversible chain with stationary distribution $\pi$, define
\[
  g(x,y) := \frac{P(x,y)}{\pi(y)} - 1,
  \qquad
  T_3(x) := \sum_{y\in S}\pi(y)\,|g(x,y)|^3.
\]
Then
\[
  \MI(X_0;X_1)
  \;\ge\;
  \frac{1}{2}\sum_{i=2}^n\lambda_i^2
  \;-\;
  \frac{1}{6}\sum_{x\in S}\pi(x)\,T_3(x).
\]
If $\pi$ is uniform, then
\[
  \bar r_\pi
  \;\ge\;
  \frac{1}{2\log n}\sum_{i=2}^n\lambda_i^2
  \;-\;
  \frac{1}{6\log n}\sum_{x\in S}\pi(x)\,T_3(x).
\]
\end{theorem}

\begin{proof}
Consider the Taylor expansion of $\varphi(u) = (1+u)\log(1+u)$ around $u=0$:
\[
  (1+u)\log(1+u) = u + \frac{u^2}{2} + R(u),
\]
where the remainder satisfies
$R(u) \ge -|u|^3/6$ for all $u > -1$
(see Lemma~\ref{lem:taylor-remainder} in Appendix~\ref{app:taylor}).

Now write
\begin{align*}
  \KL(P(x,\cdot)\|\pi)
  &= \sum_{y}\pi(y)\,(1+g(x,y))\log(1+g(x,y)) \\
  &= \sum_y \pi(y)\Bigl[g(x,y) + \frac{g(x,y)^2}{2} + R(g(x,y))\Bigr].
\end{align*}
Since $\sum_y \pi(y)\,g(x,y) = 0$ and
$\sum_y \pi(y)\,g(x,y)^2 = \chi^2(P(x,\cdot)\|\pi)$:
\[
  \KL(P(x,\cdot)\|\pi)
  = \frac{1}{2}\chi^2(P(x,\cdot)\|\pi) + \sum_y \pi(y)\,R(g(x,y))
  \ge \frac{1}{2}\chi^2(P(x,\cdot)\|\pi) - \frac{1}{6}T_3(x).
\]
Averaging over $\pi$ and applying Lemma~\ref{lem:chi2}:
\[
  \MI(X_0;X_1)
  \ge \frac{1}{2}\sum_{i=2}^n\lambda_i^2
  - \frac{1}{6}\sum_x \pi(x)\,T_3(x).
\]
The uniform-$\pi$ version follows by dividing by $\log n$.
\end{proof}

\begin{corollary}[Two-sided spectral bounds]
\label{cor:spectral-sandwich}
Combining Theorems~\ref{thm:spectral-upper} and~\ref{thm:spectral-lower},
for any reversible chain:
\[
  \frac{1}{2}\sum_{i=2}^n\lambda_i^2
  - \frac{1}{6}\sum_x\pi(x)\,T_3(x)
  \;\le\;
  \MI(X_0;X_1)
  \;\le\;
  \sum_{i=2}^n\lambda_i^2.
\]
When $\max_{x,y}|g(x,y)|\le C$ for a constant $C$ independent of $n$, one
has $\sum_x\pi(x)T_3(x) \le C\sum_{i\ge 2}\lambda_i^2$, so the lower bound
simplifies to
$\MI(X_0;X_1) \ge (1/2 - C/6)\sum_{i\ge 2}\lambda_i^2$,
which is positive when $C < 3$.
\end{corollary}

\begin{remark}[Scope of the spectral sandwich]
\label{rmk:spectral-caveat}
The two-sided bound shows that $\MI(X_0;X_1)$ and $\sum_{i\ge 2}\lambda_i^2$
are of the same order when the third-moment term is negligible compared to
$\sum_{i\ge 2}\lambda_i^2$. For chains with near-absorbing states, the ratio
$\KL/\chi^2$ at that row can become arbitrarily small, and the lower bound
may be vacuous. For instance, for a simple random walk on a $\Delta$-regular
graph with $n$ vertices, $\pi$ is uniform and
$P(x,y)/\pi(y)\in\{0,n/\Delta\}$, so
$C=\max_{x,y}|g(x,y)|=n/\Delta-1$; the bounded-row condition $C<3$ thus
requires $\Delta>n/4$. For sparse graphs (e.g., bounded-degree graphs as
$n\to\infty$), the third-moment correction dominates, and this lower bound
becomes vacuous. We state the spectral sandwich with the explicit correction
term $T_3$ rather than asserting a universal constant-factor bound to
reflect this limitation.
\end{remark}

\subsection{Eigenvector localization}

\begin{proposition}[Eigenvector localization]
\label{prop:eigenvec}
Define the $\chi^2$-based retention proxy
\[
  \tilde r(x)
  :=\frac{\chi^2(P(x,\cdot)\|\pi)}{\log(1/\pi(x))}
  =\frac{\sum_{i=2}^n\lambda_i^2\,\phi_i(x)^2}{\log(1/\pi(x))}.
\]
Then $r(x)\le\tilde r(x)$ pointwise, $\bar r_\pi\le\E_\pi[\tilde r]$,
and $M\le\max_x\tilde r(x)$.
\end{proposition}

\begin{proof}
The pointwise bound $r(x)\le\tilde r(x)$ follows from
$\KL(P(x,\cdot)\|\pi)\le\chi^2(P(x,\cdot)\|\pi)$.
Averaging over $\pi$ and taking the maxima yields the remaining claims.
\end{proof}

\begin{remark}[Interpretation]
The proxy $\tilde r(x)$ is large precisely when the squared eigenvectors
$\phi_i(x)^2$, weighted by $\lambda_i^2$, are concentrated at state~$x$
(spectral-graph eigenvector localization;
see~\cite{Chung1997} for background on spectral graph theory).
When eigenvectors are delocalized (as in vertex-transitive chains),
$\tilde r(x)$ is approximately constant.
The pointwise ordering $r(x)\le\tilde r(x)$ does \emph{not}
imply an ordering between $L(P)$ and
$\tilde L(P):=\E_\pi[\tilde r]/\max_x\tilde r(x)$,
since increasing both numerator and denominator need not preserve the ratio.
\end{remark}

\section{Cheeger-Type Lower Bound on \texorpdfstring{$M$}{M}}
\label{sec:cheeger}

\begin{theorem}[Bottleneck lower bound on $M$]
\label{thm:cheeger-M}
Let $A\subset S$ with $\pi(A)\le 1/2$ be a set achieving the Cheeger
constant~$h_P$, and assume $h_P \le 1 - \pi(A)$ (at equality the bound
below has zero numerator; see Remark~\ref{rmk:cheeger-balanced}).
Then there exists $x^*\in A$ with $P(x^*,A)\ge 1-h_P$, and
\[
  M
  \;\ge\;
  r(x^*)
  \;\ge\;
  \frac{(1-h_P)\log\dfrac{1-h_P}{\pi(A)}
  +h_P\log\dfrac{h_P}{1-\pi(A)}}{\log\!\bigl(1/\pi(x^*)\bigr)}.
\]
The term involving $h_P$ is interpreted by continuity as $0$ when $h_P=0$.
In particular, for small $h_P$,
$M \gtrsim \log(1/\pi(A))/\log(1/\pi_{\min})$.
\end{theorem}

\begin{proof}
The argument has four steps: (i) extract a state $x^*\in A$ whose row
escapes $A$ with probability at most $h_P$; (ii) reduce
$\KL(P(x^*,\cdot)\|\pi)$ to a binary KL by the data processing
inequality; (iii) replace $P(x^*,A)$ by the Cheeger lower bound
$1-h_P$ using monotonicity of binary KL; (iv) divide by
$\log(1/\pi(x^*))$ via the definition of $r(x^*)$.

\smallskip
\emph{Step 1 (existence of $x^*$).}
Since $A$ achieves the Cheeger constant,
\[
  \sum_{x\in A}\pi(x)\,P(x,A^c)\;=\;h_P\,\pi(A).
\]
Dividing by $\pi(A)>0$ exhibits $h_P$ as the $\pi|_A$-weighted average
of $P(\cdot,A^c)$ over $A$, so there exists $x^*\in A$ with
$P(x^*,A^c)\le h_P$, equivalently $P(x^*,A)\ge 1-h_P$.

\smallskip
\emph{Step 2 (data processing).}
Apply the partition map $f:S\to\{A,A^c\}$, $f(y)=\mathbf{1}_A(y)$.
The pushforwards of $P(x^*,\cdot)$ and $\pi$ are the Bernoulli laws
$\bigl(P(x^*,A),P(x^*,A^c)\bigr)$ and
$\bigl(\pi(A),1-\pi(A)\bigr)$. The data processing inequality for KL
gives
\begin{equation}\label{eq:dpi-cheeger}
  \KL\!\bigl(P(x^*,\cdot)\,\|\,\pi\bigr)
  \;\ge\;
  P(x^*,A)\log\frac{P(x^*,A)}{\pi(A)}
  +P(x^*,A^c)\log\frac{P(x^*,A^c)}{1-\pi(A)}.
\end{equation}

\smallskip
\emph{Step 3 (monotone substitution).}
Fix $q=\pi(A)$ and view the right side of \eqref{eq:dpi-cheeger} as
\[
  g(p)\;:=\;p\log\frac{p}{q}+(1-p)\log\frac{1-p}{1-q},\qquad p\in(0,1).
\]
Then $g'(p)=\log\!\bigl(p(1-q)/(q(1-p))\bigr)>0$ for $p>q$, while
$g'(q)=0$, so $g$ is nondecreasing on $[q,1]$ (with the endpoint
$p=1$ understood by continuity). The hypothesis $h_P\le 1-\pi(A)$ gives
$1-h_P\ge \pi(A)=q$, and Step~1 yields $P(x^*,A)\ge 1-h_P\ge q$, so we
are on the nondecreasing branch and may replace $p=P(x^*,A)$ by its lower
bound $1-h_P$:
\begin{equation}\label{eq:binary-kl-cheeger}
  \KL\!\bigl(P(x^*,\cdot)\,\|\,\pi\bigr)
  \;\ge\;
  (1-h_P)\log\frac{1-h_P}{\pi(A)}+h_P\log\frac{h_P}{1-\pi(A)}.
\end{equation}

\smallskip
\emph{Step 4 (divide by $\log(1/\pi(x^*))$).}
By Definition~\ref{def:r},
$r(x^*)=\KL(P(x^*,\cdot)\|\pi)/\log(1/\pi(x^*))$, and
$M\ge r(x^*)$ by Definition~\ref{def:rho}. Dividing
\eqref{eq:binary-kl-cheeger} by $\log(1/\pi(x^*))>0$ yields the
displayed bound.

\smallskip
\emph{Asymptotic.}
As $h_P\to 0$,
\[
  (1-h_P)\log\frac{1-h_P}{\pi(A)}\;\longrightarrow\;\log\frac{1}{\pi(A)},
  \qquad
  h_P\log\frac{h_P}{1-\pi(A)}\;\longrightarrow\;0.
\]
Hence the numerator tends to $\log(1/\pi(A))$, while the denominator
is at most $\log(1/\pi_{\min})$, yielding
\[
  M\;\gtrsim\;\frac{\log(1/\pi(A))}{\log(1/\pi_{\min})}.
\]
\end{proof}

\begin{corollary}[Uniform $\pi$: Cheeger bound on $M$]
\label{cor:cheeger-uniform}
If $\pi$ is uniform, the Cheeger cut has $|A|\le n/2$, and
$h_P\le 1-|A|/n$, then
\[
  M
  \;\ge\;
  \frac{(1-h_P)\log\frac{n(1-h_P)}{|A|}
  +h_P\log\frac{nh_P}{n-|A|}}{\log n}.
\]
For small $h_P$, $M\gtrsim 1-\log|A|/\log n$.
In particular, if $|A|=O(n^\alpha)$ for $\alpha\in(0,1)$, then
$M\ge 1-\alpha+o(1)$.
\end{corollary}

\begin{proof}
Immediate from Theorem~\ref{thm:cheeger-M} with $\pi(x)=1/n$,
$\pi(A)=|A|/n$, and $\log(1/\pi(x^*))=\log n$.
\end{proof}

\begin{remark}[Limitation for balanced bottlenecks]
\label{rmk:cheeger-balanced}
Corollary~\ref{cor:cheeger-uniform} is strongest for small bottlenecks
($|A|\ll n$), where it yields $M\to 1$. However, for balanced cuts in the
uniform setting ($|A|=n/2$), the displayed lower bound is at most
$\log 2/\log n$ for every $h_P\in[0,1/2]$ (with the endpoint $h_P=0$
interpreted by continuity), and therefore tends to zero as $n\to\infty$
regardless of how small $h_P$ is. This reflects the fact that when the
Cheeger cut splits the state space into two halves of comparable size, the
resulting obstruction is global rather than localized, and cannot be detected
sharply by a single high-retention state.
\end{remark}

\section{The Convexity Gap and Near-Optimality of Point Masses}
\label{sec:convexity-gap}

We address the question: \emph{how much can collective initial distributions
outperform point masses in achieving the supremum $\etaKL(P)$?} We identify
the mechanism as the convexity gap of KL divergence and express it in terms
of mutual information.

\begin{definition}[Mutual information for a general initial distribution]
\label{def:mi-general}
For any probability distribution $\mu$ on $S$, let $X\sim\mu$ and
$Y\mid X=x\sim P(x,\cdot)$. Define
\[
  I_\mu(X;Y)
  := \sum_{x\in S}\mu(x)\,\KL\bigl(P(x,\cdot)\,\|\,\mu P\bigr),
\]
where $\mu P$ is the marginal of $Y$. This equals the standard Shannon
mutual information between $X$ and $Y$ under the joint distribution
$\mu(x)P(x,y)$.
\end{definition}

\begin{definition}[Convexity gap]
For a distribution $\mu$ on $S$, define
\[
  \Delta(\mu)
  := \sum_{x\in S}\mu(x)\,\KL(P(x,\cdot)\|\pi)
     - \KL(\mu P\|\pi).
\]
By the convexity of $\nu\mapsto\KL(\nu\|\pi)$ (Jensen's inequality),
$\Delta(\mu)\ge 0$.
\end{definition}

\begin{theorem}[Convexity gap identity]
\label{thm:convexity-gap}
For any probability distribution $\mu$ on $S$,
\[
  \Delta(\mu)
  = \sum_{x\in S}\mu(x)\,\KL(P(x,\cdot)\|\mu P)
  = I_\mu(X;Y).
\]
(Identity of this form---sometimes called the ``compensation
identity'' or the chain-rule expression for mutual information---is
standard in information theory; see, e.g.,
Polyanskiy--Wu~\cite[Ch.~4]{PolyanskiyWu2025} and
Csisz\'ar--Shields~\cite{CsiszarShields2004}. We record the short proof in
the form needed below and use it to control the gap
$\etaKL(P)-M$ in Theorem~\ref{thm:central-bound}.)
\end{theorem}

\begin{proof}
For any $x\in S$, the difference of KL divergences with different second arguments is:
\begin{align*}
  \Delta(\mu)
  &= \sum_x\mu(x)\Bigl[\KL(P(x,\cdot)\|\pi) - \KL(P(x,\cdot)\|\mu P)\Bigr] \\
  &\quad + \sum_x\mu(x)\KL(P(x,\cdot)\|\mu P) - \KL(\mu P\|\pi).
\end{align*}

The first sum equals
\begin{align*}
  \sum_x\mu(x)\sum_y P(x,y)\log\frac{(\mu P)(y)}{\pi(y)}
  &= \sum_y(\mu P)(y)\log\frac{(\mu P)(y)}{\pi(y)} \\
  &= \KL(\mu P\|\pi).
\end{align*}

Substituting:
\begin{align*}
  \Delta(\mu)
  &= \KL(\mu P\|\pi) + \sum_x\mu(x)\KL(P(x,\cdot)\|\mu P) - \KL(\mu P\|\pi) \\
  &= \sum_x\mu(x)\KL(P(x,\cdot)\|\mu P) \\
  &= I_\mu(X;Y). \qedhere
\end{align*}
\end{proof}

\begin{theorem}[Central bound: mutual information controls the gap
$\etaKL(P) - M$]
\label{thm:central-bound}
For any $\mu\neq\pi$,
\[
  \frac{\KL(\mu P\|\pi)}{\KL(\mu\|\pi)}
  \;\le\;
  M\cdot\frac{\KL(\mu\|\pi)+H(\mu)}{\KL(\mu\|\pi)}
  \;-\;
  \frac{I_\mu(X;Y)}{\KL(\mu\|\pi)},
\]
where $H(\mu)=-\sum_x\mu(x)\log\mu(x)$ is the Shannon entropy.
Equality holds if and only if $r(x)=M$ for all $x$ in the support of $\mu$.
\end{theorem}

\begin{proof}
From Theorem~\ref{thm:convexity-gap},
$\KL(\mu P\|\pi) = \sum_x \mu(x)\KL(P(x,\cdot)\|\pi) - I_\mu(X;Y)$.
Since $r(x)\le M$,
\[
  \sum_x \mu(x)\KL(P(x,\cdot)\|\pi)
  = \sum_x \mu(x)\,r(x)\log(1/\pi(x))
  \le M\sum_x \mu(x)\log(1/\pi(x)).
\]
The identity
$\sum_x\mu(x)\log(1/\pi(x))
= \KL(\mu\|\pi)+H(\mu)$
gives
\[
  \KL(\mu P\|\pi)
  \le M(\KL(\mu\|\pi)+H(\mu)) - I_\mu(X;Y).
\]
Dividing by $\KL(\mu\|\pi)$ yields the bound. Equality holds when
$r(x)\le M$ is tight for all $x$ in $\mathrm{supp}(\mu)$, i.e., $r(x)=M$
on the support.
\end{proof}

\begin{remark}[Interpretation of the central bound]
\label{rmk:central-interpretation}
The bound has two terms pulling in opposite directions. The first,
$M\cdot H(\mu)/\KL(\mu\|\pi)$, is an \emph{entropy inflation}: spreading
$\mu$ increases the cross-entropy
$\sum_x\mu(x)\log(1/\pi(x))=\KL(\mu\|\pi)+H(\mu)$ relative to the KL
denominator, and so loosens the bound. The second,
$I_\mu(X;Y)/\KL(\mu\|\pi)$, enters with a minus sign and tightens it:
any $\mu$ supported on states with distinct rows forces $\mu P$ to
retain row-specific structure, which costs positive mutual information.
The two effects are charged against each other, so a candidate $\mu$
can outperform a point mass only when entropy inflation gains beat the
$I_\mu$ penalty.

The point-mass case is degenerate: $H(\delta_x)=0$ and
$I_{\delta_x}(X;Y)=0$ recover $r(x)\le M$ exactly. The opposite limit
$\mu\to\pi$ is also degenerate, in the other direction:
$H(\mu)/\KL(\mu\|\pi)$ blows up and the bound stops being informative,
which is why the local quadratic estimate of
Proposition~\ref{prop:local-l2-lower} is needed near stationarity. The
two bounds cover non-overlapping regimes.

Read through $L(P)$: when $L(P)$ is small the high-$r$ set carries
little $\pi$-mass, so distributions concentrated on it have little
entropy to inflate with, and distributions spread over the rest of the
space pay an $I_\mu$ penalty unless their rows already coincide. This
is a heuristic for point-mass near-optimality, not a theorem;
Appendix~\ref{app:counterexample} shows it can fail when many low-mass
states share identical rows. When $L(P)\approx 1$ the high-$r$ set is
the whole space, the entropy term is no longer constrained, and
collective or near-stationary perturbations can win, as they do for
the path, cycle, hypercube, and barbell.
\end{remark}

\begin{corollary}[Uniform-$\pi$ contraction bound]
\label{cor:eta-uniform-bound}
If $\pi$ is uniform on $S$ with $|S|=n$, then for any distribution
$\mu$ on $S$,
\[
  \KL(\mu P\|\pi)
  \;\le\;
  M\,\log n \;-\; I_\mu(X;Y).
\]
\end{corollary}

\begin{proof}
When $\pi$ is uniform,
$\sum_x\mu(x)\log(1/\pi(x)) = \log n$ for every $\mu$,
so $\KL(\mu\|\pi) + H(\mu) = \log n$.
Substituting into Theorem~\ref{thm:central-bound}
(before dividing by $\KL(\mu\|\pi)$) gives the stated bound.
\end{proof}

\subsection*{Numerical methodology}
\label{sec:numerical-methodology}

The numerical experiments are exploratory and are used to calibrate the
framework, not to certify exact values of $\etaKL(P)$. For a fixed chain, the
optimized objective is
\[
  F(\mu)=\frac{\KL(\mu P\|\pi)}{\KL(\mu\|\pi)},
  \qquad \mu\neq\pi.
\]
The search aggregates the best ratio found across six candidate classes:
\begin{enumerate}
\item \emph{Point masses.} Each $\delta_x$, $x\in S$, is evaluated explicitly,
  since the softmax parameterization below does not represent the simplex
  boundary exactly.
\item \emph{Near-stationary eigenfunction candidates.} For each nontrivial
  eigenpair $(\lambda_k,\phi_k)$ of the $L^2(\pi)$-symmetrization
  \[
    A \;=\; \mathrm{diag}(\sqrt\pi)\,P\,\mathrm{diag}(\sqrt\pi)^{-1},
  \]
  after rescaling by $1/\sqrt\pi$, mean-centering against $\pi$, and
  $L^2(\pi)$-normalizing, we test
  $\mu_\varepsilon=\pi\odot(1+\varepsilon\phi_k)$ for
  $\varepsilon\in\{0.05,\,0.1,\,0.2,\,0.4,\,0.6,\,0.8\}$, skipping any
  $\varepsilon$ that produces a non-positive coordinate. These calibrate
  $F(\mu_\varepsilon)$ against the analytic limit
  \[
    F(\mu_\varepsilon)\;\longrightarrow\;
    \|P\phi_k\|_{L^2(\pi)}^2\,/\,\|\phi_k\|_{L^2(\pi)}^2
  \]
  predicted by Proposition~\ref{prop:local-l2-lower}.
\item \emph{Family-specific structured mixtures}, e.g.\ arcs on paths
  and cycles, clique-supported distributions on barbells, and
  Hamming-ball or face distributions on hypercubes.
\item \emph{Dirichlet random initializations:} $120$ samples per chain
  with concentration $\alpha$ chosen uniformly per draw from
  $\{0.1,\,0.3,\,1,\,3,\,10\}$, covering both diffuse and concentrated
  regimes.
\item \emph{Heavy-tailed softmax initializations:} $60$ samples
  $\theta\sim\mathcal N(0,\sigma^2 I_n)$ with $\sigma$ chosen uniformly
  per draw from $\{0.5,\,2,\,5,\,15\}$.
\item \emph{Smooth interior optimization} on the simplex via the softmax map
  \[
    \mu_i(\theta)=\frac{e^{\theta_i}}{\sum_j e^{\theta_j}},
  \]
  using L-BFGS-B from \texttt{scipy.optimize.minimize} with
  \texttt{maxiter}\,$=5000$, \texttt{ftol}\,$=10^{-15}$, and
  \texttt{gtol}\,$=10^{-12}$. For $n\le 30$ a Nelder--Mead pass with
  \texttt{maxiter}\,$=10^4$, \texttt{xatol}\,$=10^{-12}$, and
  \texttt{fatol}\,$=10^{-14}$ is run from the same initializations as a
  cross-check.
\end{enumerate}
The objective $F$ is singular at $\mu=\pi$: the implementation returns
$0$ whenever $\KL(\mu\|\pi)<10^{-12}$, and KL evaluations clip
masses below $10^{-300}$ to avoid $\log 0$. Boundary point masses are
evaluated directly (see item~1) rather than approximated through softmax.
The near-stationary values are calibrated against the analytic limit of
Proposition~\ref{prop:local-l2-lower} via the eigenfunction candidates of
item~2.

In the reproducibility audit accompanying this manuscript, the random seed
is fixed at $20260511$, and the numerical environment is Python~3.12.13
with NumPy~2.3.5. The reported table values are the best lower bounds
found over these candidate classes; the driver script is
\texttt{reproducibility/}\allowbreak\texttt{verify\_table\_}\allowbreak\texttt{remark\_7\_6.py},
with results serialized to
\texttt{verify\_table\_}\allowbreak\texttt{remark\_7\_6\_}\allowbreak\texttt{results.json}.
Unless an analytic argument is stated separately, the entries should be
interpreted as numerical lower bounds on $\etaKL(P)/M$, not as certified
suprema.

\begin{remark}[Empirical behavior of $\etaKL(P)/M$]
\label{rmk:near-opt}
The following table summarizes the best numerical lower bounds found for
$\etaKL(P)/M$ in the tested chain families. The table is intended to describe
observed behavior in these examples, not to assert universal optimality of any
candidate class.

\begin{center}
\renewcommand{\arraystretch}{1.15}
\begin{tabular}{@{}lcccc@{}}
\toprule
Chain family & $|S|$ tested & $L(P)$ & $\etaKL(P)/M$ & Character \\
\midrule
Complete graph $K_n$ & $4$--$20$ & $1.00$ & $1.00$ & Point-mass benchmark \\
Star $S_n$ & $5$--$50$ & $0.50$ & $1.00^{\dagger}$ & No improvement found \\
Biased BD ($\lambda=0.85$)$^{\ddagger}$ & $10$--$40$ & $0.24$--$0.27$ & $1.01$--$1.04$ & Near point-mass \\
Two-state symmetric$^{\ast}$ & $2$ & $1.00$ & $1.21$--$1.37$ & Local spectral optimum \\
Hypercube $Q_d$ & $4$--$128$ & $1.00$ & $1.00$--$1.13$ & Mild gap \\
Lazy cycle $C_n$ & $10$--$50$ & $1.00$ & $1.35$--$1.49$ & Local quadratic regime \\
Lazy path $P_n$$^{\S}$ & $10$--$50$ & $0.87$--$0.95$ & $1.28$--$1.47$ & Local quadratic regime \\
Barbell $B_m$ & $10$--$18$ & $0.96$--$0.99$ & $2.33$--$2.44$ & Strong local quadratic regime \\
\bottomrule
\end{tabular}
\end{center}
\smallskip\noindent
{\small Sources: rows for the complete graph $K_n$ and the symmetric
two-state chain are analytic ($\etaKL/M$ in closed form); all other
entries are numerical lower bounds from the multi-start search of
\S\ref{sec:numerical-methodology}.\\
$^{\dagger}$Based on numerical optimization, not analytic proof.\\
$^{\ddagger}$Biased birth--death chain on $\{0,1,\dots,n{-}1\}$ with bias
parameter $\lambda\in(0,1)$. Interior states ($1\le x\le n{-}2$) are lazy:
$P(x,x)=1/2$, $P(x,x{+}1)=\lambda/2$, $P(x,x{-}1)=(1-\lambda)/2$.
Boundary states are reflecting and non-lazy:
$P(0,0)=1-\lambda$, $P(0,1)=\lambda$,
$P(n{-}1,n{-}1)=\lambda$, $P(n{-}1,n{-}2)=1-\lambda$.
The stationary distribution is proportional to
\[
  w_0=1,\qquad
  w_x=2\Bigl(\frac{\lambda}{1-\lambda}\Bigr)^x\quad(1\le x\le n-2),
  \qquad
  w_{n-1}=\Bigl(\frac{\lambda}{1-\lambda}\Bigr)^{n-1},
\]
with normalization $\pi(x)=w_x/\sum_y w_y$.\\
$^{\S}$Non-reflecting boundary: endpoints transition to their unique
neighbor with probability~$1/2$, yielding a non-uniform stationary
distribution.\\
$^{\ast}$For the symmetric two-state lazy chain with off-diagonal
probability $a\in(0,1/2]$, the eigenvalues are $\{1,1-2a\}$ and a direct
computation gives the closed form $\etaKL(P)=(1-2a)^2$ (cf.\
Ahlswede--G\'acs~\cite{AhlswedeGacs1976}); this matches the local
quadratic value of Proposition~\ref{prop:local-l2-lower} and exceeds the
point-mass ratio for the tested values $a\in\{0.1,0.2,0.35\}$.}
The computations support the qualitative distinction suggested by
Theorem~\ref{thm:central-bound} and Proposition~\ref{prop:local-l2-lower}:
low-$L$ examples in this test suite were close to the point-mass lower
bound, whereas several high-$L$ examples were controlled by collective or
near-stationary perturbations. This is evidence for the heuristic, not a
classification theorem.
\end{remark}

\begin{remark}[Limits of $L(P)$ as a diagnostic of $\etaKL(P)/M$]
\label{rmk:rho-counterexample}
One might conjecture that $L(P_n)\to 0$ implies $\etaKL(P_n)/M_n\to 1$.
Appendix~\ref{app:counterexample} provides a formal construction proving this
is false: a family of chains with $L(P)\to 0$ but
$\etaKL(P)/M\ge 2-o(1)$. The mechanism is that $L(P)$ controls the
total $\pi$-mass of high-$r$ states but not their cardinality; when many
low-mass states share identical rows, the uniform mixture has
$I_\mu(X;Y)=0$ while entropy inflation drives the contraction ratio
beyond any point mass.
\end{remark}

\section{Mixing Time and Contraction Rate}
\label{sec:mixing}

We record the mixing-time consequences of KL contraction. This bound is
not competitive with the sharpest spectral or coupling bounds (see
\cite{LevinPeres2017,MontenegroTetali2006} for comprehensive overviews); its
purpose is to connect the row-based quantities $r(x)$, $M$, and $L(P)$ to
mixing time, thus completing the structural analysis.

\begin{theorem}[Mixing time upper bound via KL contraction]
\label{thm:mixing-upper}
For any ergodic chain and any $\varepsilon\in(0,1)$, if
$0<\etaKL(P)<1$, then
\[
  t_{\mathrm{mix}}(\varepsilon)
  \;\le\;
  \max\left\{0,\,
  \left\lceil
  \frac{\log\!\bigl(\frac{1}{2\varepsilon^2}\log(1/\pi_{\min})\bigr)}{-\log \etaKL(P)}
  \right\rceil
  \right\}.
\]
If $\etaKL(P)=0$, then $t_{\mathrm{mix}}(\varepsilon)\le1$.
\end{theorem}

\begin{proof}
If $\etaKL(P)=0$, then $\KL(\mu P\|\pi)=0$ for every initial distribution
$\mu$, hence $\mu P=\pi$ after one step. Assume now that $0<\etaKL(P)<1$.
From Proposition~\ref{prop:multi-step} (Appendix~\ref{app:contraction}):
$\KL(\delta_x P^t\|\pi) \le \etaKL(P)^t\log(1/\pi(x))$.
Applying Pinsker's inequality
$\|\mu-\nu\|_{\TV}^2\le\frac{1}{2}\KL(\mu\|\nu)$
and using $\log(1/\pi(x))\le\log(1/\pi_{\min})$:
\[
  \max_{x\in S}\|P^t(x,\cdot)-\pi\|_{\TV}
  \;\le\;
  \sqrt{\tfrac{1}{2}\,\etaKL(P)^t\,\log(1/\pi_{\min})}.
\]
We require this total-variation distance to be at most~$\varepsilon$.
Squaring both sides and rearranging yields
\[
  \etaKL(P)^t \;\le\; \frac{2\varepsilon^2}{\log(1/\pi_{\min})}.
\]
Taking the natural logarithm, and noting that
$\log\etaKL(P)<0$ since $\etaKL(P)<1$, the inequality reverses to give
\[
  t \;\ge\;
  \frac{\log\!\Bigl(\frac{\log(1/\pi_{\min})}{2\varepsilon^2}\Bigr)}
       {-\log\etaKL(P)}.
\]
Taking the ceiling and then the maximum with~$0$ to enforce a non-negative
number of steps yields the stated bound.
\end{proof}

\begin{remark}[Role of $M$ and $L(P)$]
From the hierarchy $\etaKL(P)\ge M=\bar r_\pi/L(P)$, large $M$ forces
$\etaKL(P)$ close to~$1$ and weakens the convergence guarantee.
For fixed $\bar r_\pi$, smaller $L(P)$ inflates $M$, formalizing the
intuition that localized obstructions impede worst-case convergence.
The hierarchy is one-sided: $M$ constrains $\etaKL(P)$ from below but not
from above.
\end{remark}

\begin{remark}[Inherent log-log structure]
The $\log\log(1/\pi_{\min})$ growth in the numerator is intrinsic to the
KL--Pinsker approach: the initial KL distance $\log(1/\pi_{\min})=O(\log n)$
requires only $O(\log\log n)$ steps of geometric contraction before Pinsker's
inequality certifies TV convergence. This makes the bound weaker than the
$O(\log n/\gamma)$ spectral bounds. The trade-off is computational:
$r(x)$, $M$, and $L(P)$ are computable directly from the kernel without
eigenvector decomposition.
\end{remark}

\section{Examples}
\label{sec:examples}

\begin{proposition}[Complete graph]
\label{prop:complete}
Consider the simple random walk on $K_n$:
$P(x,\cdot)=\text{Uniform on }S\setminus\{x\}$.
Then $\pi$ is uniform,
$H(P(x,\cdot))=\log(n-1)$,
$\etaKL(P)\ge 1-\log(n-1)/\log n$,
$r(x)$ is constant, and $L(P)=1$.
\end{proposition}

\begin{proof}
Each row $P(x,\cdot)$ is uniform on $n-1$ states, so
$H(P(x,\cdot))=\log(n-1)$. Apply Corollary~\ref{cor:uniform}.
By vertex-transitivity (Proposition~\ref{prop:vertex-trans}), $r(x)$ is
constant, $M=\bar r_\pi$, and $L(P)=1$.
\end{proof}

\begin{proposition}[Lazy cycle]
\label{prop:cycle}
For the lazy random walk on the cycle $C_n$, $n\ge3$, with
$P(x,x)=1/2$, $P(x,x\pm 1)=1/4$,
the stationary distribution is uniform,
$H(P(x,\cdot)) = \frac{3}{2}\log 2$,
and $L(P)=1$ by vertex-transitivity.
\end{proposition}

\begin{proof}
\[
  H(P(x,\cdot))=-\Bigl(\frac{1}{2}\log\frac{1}{2}+2\cdot\frac{1}{4}\log\frac{1}{4}\Bigr)
  =\frac{1}{2}\log 2+\frac{1}{2}\cdot 2\log 2
  =\frac{3}{2}\log 2.
\]
Vertex-transitivity (Proposition~\ref{prop:vertex-trans}) gives $L(P)=1$.
\end{proof}

\begin{proposition}[Vertex-transitive chains]
\label{prop:vertex-trans}
If $P$ arises from a random walk on a vertex-transitive graph, then $r(x)$ is
constant in $x$; hence, $L(P)=1$.
\end{proposition}

\begin{proof}
For any automorphism $\sigma$ of $(S,P)$ preserving $\pi$,
$P(\sigma(x),\sigma(y))=P(x,y)$ and $\pi(\sigma(x))=\pi(x)$. Therefore,
$\KL(P(\sigma(x),\cdot)\|\pi) = \KL(P(x,\cdot)\|\pi)$
and $\log(1/\pi(\sigma(x)))=\log(1/\pi(x))$, so
$r(\sigma(x))=r(x)$. Transitivity forces $r$ to be constant, whence
$\bar r_\pi=M$ and $L(P)=1$.
\end{proof}

\begin{remark}[Vertex-transitive chains span all mixing speeds]
\label{rmk:vertex-trans-speed}
Proposition~\ref{prop:vertex-trans} implies $L(P)=1$ for all
vertex-transitive chains, whether $\gamma=\Theta(1)$ (complete graph) or
$\gamma=\Theta(1/n^2)$ (cycle). Thus $L(P)$ is
\emph{structurally decoupled} from $\gamma$: no functional
relationship can hold across vertex-transitive chains.
\end{remark}

\begin{proposition}[Star graph]
\label{prop:star}
For the lazy random walk on the star graph with $n\ge3$ vertices, labelled
so that vertex~$0$ is adjacent to vertices $1,\ldots,n-1$ and there are no
other non-loop edges, the stationary distribution satisfies
$\pi(0)=1/2$ and $\pi(i)=1/(2(n-1))$ for $i\ge 1$.
In addition, $L(P)=1/2$.
\end{proposition}

\begin{proof}
The transition matrix satisfies
$P(0,0)=1/2$, $P(0,i)=1/(2(n-1))$ for $i\ge 1$,
and $P(i,i)=P(i,0)=1/2$ for each $i\ge 1$.

\medskip
\noindent\textbf{State $0$.}
$P(0,\cdot)=\pi$, hence $\KL(P(0,\cdot)\|\pi)=0$ and $r(0)=0$.

\medskip
\noindent\textbf{States $i\ge 1$.}
For any $i\ge 1$:
\[
  \KL(P(i,\cdot)\|\pi)
  =\frac{1}{2}\log\frac{1/2}{\pi(i)}
  +\frac{1}{2}\log\frac{1/2}{\pi(0)}
  =\frac{1}{2}\log(n-1)+\frac{1}{2}\log 1
  =\frac{1}{2}\log(n-1),
\]
and $\log(1/\pi(i))=\log(2(n-1))$, so
$r(i)=\frac{\frac{1}{2}\log(n-1)}{\log(2(n-1))}$.
Since $r(0)=0$ and $r(i)>0$ for $i\ge 1$ when $n\ge 3$,
$M=r(i)$ for any $i\ge 1$. Finally,
$\bar r_\pi
=\pi(0)\cdot 0+\sum_{i=1}^{n-1}\pi(i)r(i)
=\frac{1}{2}\,M$,
so $L(P)=1/2$.
\end{proof}

\begin{remark}[Near-optimality on the star: numerical evidence]
Numerical optimization did not identify distributions outperforming point
masses for the lazy star for tested values up to $n=50$. This remains
numerical evidence, not an analytic proof that $\etaKL(P)=M$ holds for all
$n$. The pattern is consistent with Theorem~\ref{thm:central-bound}: the only
high-$r$ states are $1,\ldots,n-1$, all of which satisfy $r(i)=M$. Although
these states have distinct rows (differing in the self-loop state), mixing
over them introduces positive mutual information $I_\mu(X;Y)$ that appears to
offset the entropy inflation in the tested cases.
\end{remark}

\begin{remark}[Barbell graph]
\label{rmk:barbell}
The lazy random walk on the barbell graph (two $K_m$ cliques connected by a
single edge) has $L(P)\approx 0.96$--$0.99$ for $m=5$--$9$ in the tested
range. This is consistent with the interpretation that near-worst row
retention is spatially widespread, even though the conductance bottleneck is
localized at the bridge. Meanwhile, the best numerical lower bounds give
$\etaKL(P)/M\approx 2.3$--$2.4$. In these examples, the local quadratic value
from Proposition~\ref{prop:local-l2-lower} already accounts for the large gap,
so the computation should be read as evidence of a global/near-stationary
obstruction rather than as a certified optimizer description.
\end{remark}

\begin{remark}[Hypercube]
\label{rmk:hypercube}
The hypercube row in the table uses the standard \emph{random-coordinate} lazy
walk on $Q_d=\{0,1\}^d$: with probability~$1/2$ the walker stays put, and
otherwise a uniformly chosen coordinate is flipped, i.e.
\[
  P \;=\; \tfrac{1}{2}I \;+\; \frac{1}{2d}\sum_{i=1}^d Q_i,
\]
where $Q_i$ exchanges the $i$-th bit. This is an additive---not
multiplicative---combination of the single-coordinate flips, so it is
\emph{not} the tensor product of the lazy walk on $K_2$: the latter has
$P_{K_2}=\bigl(\begin{smallmatrix}1/2&1/2\\1/2&1/2\end{smallmatrix}\bigr)$,
whose rows already coincide with~$\pi_{K_2}$, so $M_{K_2}=0$ and the
tensor-product chain $P_{K_2}^{\otimes d}$ has every row equal to
$\pi_{Q_d}$ (a degenerate, one-step-mixing chain). This random-coordinate
walk, by contrast, is invariant under the action of the
hyperoctahedral group on $Q_d$, hence vertex-transitive, so
Proposition~\ref{prop:vertex-trans} gives $L(Q_d)=1$. A direct
row computation yields
\[
  M_{Q_d} \;=\; 1 \;-\; \frac{\log(4d)}{2d\log 2},
\]
which matches the table entries for $d=2,\ldots,7$
($|S|=4,\ldots,128$).
\end{remark}

\begin{remark}[Comparing $L(P)$ and $\gamma$]
\label{rmk:rho-vs-gap}
The following comparisons illustrate how $L(P)$ and the spectral gap~$\gamma$
capture different aspects of chain structure.

\medskip
\noindent\textbf{Same $L$, different $\gamma$.}
The complete graph $K_n$ and the lazy cycle $C_n$ both satisfy $L(P)=1$
(Propositions~\ref{prop:complete} and~\ref{prop:cycle}): In both chains,
each state exhibits the same local retention. Yet
$\gamma(K_n)=(n-2)/(n-1)=\Theta(1)$ while $\gamma(C_n)=\Theta(1/n^2)$.
Thus $L(P)$ does not reflect mixing speed: it identifies that the
contraction obstruction is spatially uniform in both cases, but says nothing
about how severe it is.

\medskip
\noindent\textbf{Same $\gamma$, different $L$.}
The reverse distinction can already be seen on three states. On
$S=\{0,1,2\}$, consider the two transition matrices
\[
  P_A=
  \begin{pmatrix}
  0 & 1/2 & 1/2\\
  1/2 & 0 & 1/2\\
  1/2 & 1/2 & 0
  \end{pmatrix},
  \qquad
  P_B=
  \begin{pmatrix}
  1/2 & 1/4 & 1/4\\
  1/2 & 1/2 & 0\\
  1/2 & 0 & 1/2
  \end{pmatrix}.
\]
The first chain moves uniformly to one of the other two states. In the
second chain, state~$0$ has transition row $(1/2,1/4,1/4)$, while states
$1$ and~$2$ have transition rows $(1/2,1/2,0)$ and $(1/2,0,1/2)$. Their
stationary distributions are
$\pi_A=(1/3,1/3,1/3)$ and
$\pi_B=(1/2,1/4,1/4)$, and their nontrivial eigenvalues are
$\{-1/2,-1/2\}$ and $\{1/2,0\}$, respectively. Hence, both chains have
absolute spectral gap $\gamma=1/2$. However, their row-retention profiles
are different. For $P_A$, each row is uniform on two states, so
$\KL(P_A(x,\cdot)\|\pi_A)=\log(3/2)$ and
$\KL(\delta_x\|\pi_A)=\log 3$. For $P_B$, the row from state~$0$
equals $\pi_B$, while for either state $i\in\{1,2\}$,
$\KL(P_B(i,\cdot)\|\pi_B)=\frac12\log 2$ and
$\KL(\delta_i\|\pi_B)=\log 4$. Writing $r_A$ and $r_B$ for the corresponding
retention profiles, this gives
\[
  r_A(x)=\frac{\log(3/2)}{\log 3}\quad\text{for all }x,
  \qquad
  r_B(0)=0,\qquad r_B(1)=r_B(2)=\frac14.
\]
Consequently $L(P_A)=1$ whereas $L(P_B)=1/2$. The spectral gap
assigns the same $L^2$ contraction scale to the two chains, but the retention
profile distinguishes the mechanism: the one-step KL obstruction is global
for $P_A$ and concentrated on states~$1$ and~$2$ for $P_B$.

\medskip
\noindent\textbf{Different $L$, both polynomial mixing.}
The lazy cycle $C_n$ has $L(P)=1$: the retention profile is constant,
and slowness is spread uniformly across all states. The lazy star chain of
Proposition~\ref{prop:star} has $L(P)=1/2$: state~$0$ retains nothing
($r(0)=0$, since $P(0,\cdot)=\pi$), while states $1,\ldots,n-1$ carry the
full row-level obstruction ($r(i)=M$). Both chains mix in polynomial time,
and $\gamma$ differs between them. What $\gamma$ does not directly reveal is
the spatial distribution of the contraction obstruction---whether slowness is
homogeneous or concentrated on a subset of states. This is the information
that $L(P)$ encodes.
\end{remark}

\section{Discussion}
\label{sec:discussion}

\subsection*{Summary of contributions}

The primary contributions are the framework itself, together with three new
results built on it:

\begin{enumerate}
\item \textbf{Convexity-gap identity and contraction-ratio decomposition}
  (Theorems~\ref{thm:convexity-gap}--\ref{thm:central-bound}). The gap
  between the row-averaged divergence and $\KL(\mu P\|\pi)$ equals the
  mutual information $I_\mu(X;Y)$, yielding a decomposition of
  $\KL(\mu P\|\pi)/\KL(\mu\|\pi)$ into entropy inflation minus a
  mutual-information penalty. This identifies the precise mechanism by
  which collective initial distributions can outperform point masses.

\item \textbf{Cheeger-type lower bound on $M$}
  (Theorem~\ref{thm:cheeger-M}). A geometric bottleneck at a set of small
  $\pi$-mass forces a high-retention row, connecting bottleneck geometry
  directly to the row-retention profile.

\item \textbf{Limits of $L(P)$ as a diagnostic}
  (Theorem~\ref{thm:counterexample}, Appendix~\ref{app:counterexample}). An
  explicit construction with $L(P)\to 0$ but $\etaKL/M\ge 2-o(1)$,
  isolating the cardinality (not the $\pi$-mass) of high-retention states
  as the decisive quantity.
\end{enumerate}

The framework also yields the following structural consequences, several of
which are direct applications of standard tools to the retention profile:

\begin{enumerate}
\setcounter{enumi}{3}
\item \textbf{Discrete-time entropy contraction equivalence}
  (Remark~\ref{rmk:mlsi-equivalence}). For finite reversible chains,
  $\etaKL(P)=1-\alpha_1(P)$, where $\alpha_1(P)$ is the discrete-time
  full-step entropy contraction gap; the row-based lower bound
  $\etaKL(P)\ge M$ is equivalent to evaluating the entropy contraction
  ratio at a point-mass test function (Remark~\ref{rmk:mlsi}).

\item \textbf{Optimal tail bounds}
  (Theorem~\ref{thm:ratio-characterization-general}). Markov and
  reverse-Markov bounds for the retention profile, optimal given only the
  mean and support constraints.

\item \textbf{Variance bound} (Proposition~\ref{prop:variance}). An
  instance of the Bhatia--Davis inequality~\cite{BhatiaDavis2000} applied
  to $r$, showing that $L(P)\approx 1$ forces approximate constancy of the
  retention profile and quantifies the remaining dispersion.

\item \textbf{Two-sided spectral bounds}
  (Theorems~\ref{thm:spectral-upper} and~\ref{thm:spectral-lower}). The
  mutual information $\MI(X_0;X_1)$ is sandwiched between
  $\frac{1}{2}\sum\lambda_i^2 - \frac{1}{6}\sum\pi(x)T_3(x)$ and
  $\sum\lambda_i^2$, with the gap controlled by the third-moment term
  $T_3$.

\item \textbf{Tensorization} (Appendix~\ref{sec:tensor}).
  For product chains, the local retention profile is a log-weighted average
  of the factor profiles, implying $M_P\le\max(M_1,M_2)$; in the identical
  uniform two-fold case, tensorization weakly increases $L(P)$.
\end{enumerate}

\subsection*{Connections to classical theory}

$L(P)$ does not replace the spectral gap, the Cheeger constant, or the
log-Sobolev constant; it sits next to them. The spectral upper bound
(Theorem~\ref{thm:spectral-upper}) controls $\bar r_\pi$ via the
$\KL\le\chi^2$ comparison and the eigenvalues of the symmetrization,
and Theorem~\ref{thm:spectral-lower} gives a partial converse with an
explicit cubic correction in the row entries. The Cheeger-type bound
(Theorem~\ref{thm:cheeger-M}) shows that a small-mass set with low
conductance forces a single high-$r$ row, so geometric bottlenecks
show up directly in the retention profile rather than only through the
spectrum.

\subsection*{Empirical relationship to classical invariants}

Numerical experiments on a test suite of $46$ chains across $12$
structured families (complete, lazy complete, cycle, hypercube, star,
path, barbell, lollipop, two-cluster, doubly stochastic, biased
birth--death, Ehrenfest) produced the following descriptive Spearman
rank correlations:
\[
  r_s(L,\gamma)\approx +0.003,
  \qquad
  r_s(L,h_P)\approx +0.10,
  \qquad
  r_s(L,t_{\mathrm{mix}})\approx -0.11.
\]
Within the tested chain families, no clear monotone relationship was visible.
These numbers should be interpreted as descriptive summaries of this particular
suite, not as inferential statistics: the examples are highly structured and
are not a random sample from a well-defined population of chains. The
correlations are computed by
\texttt{reproducibility/}\allowbreak\texttt{simulate\_rho\_}\allowbreak\texttt{spectral\_}\allowbreak\texttt{cheeger.py}
in the accompanying reproducibility bundle.

Analytically, Proposition~\ref{prop:vertex-trans} shows that every
vertex-transitive chain satisfies $L(P)=1$, whether $\gamma=\Theta(1)$ or
$\gamma=\Theta(1/n^2)$, so no functional relationship between
$L(P)$ and $\gamma$ can hold across this class.

\subsection*{Localization heuristic}

When $L(P)$ is small, slow contraction is concentrated on a few
high-$r$ states. Efficient initial distributions then tend to put their
mass on those states, which keeps $H(\mu)$ small and makes point masses
or near-point masses comparatively effective. When $L(P)$ is close to
$1$, near-worst retention is spread across much of the state space, and
collective distributions can exploit the geometry of the bottleneck;
the ratio $\etaKL(P)/M$ may then be substantially above~$1$, as the
barbell, path, cycle, and hypercube examples show.

\subsection*{Practical implications for MCMC design}

The quantities $M$, $L(P)$, and the profile $r(x)$ are computable
directly from the transition kernel, without an eigendecomposition,
and they isolate \emph{where} on the state space the contraction
obstruction lives. When $L(P)$ is small, the obstruction is carried by
a few high-$r$ states; locating these points suggests targeted remedies
(local moves near the bottleneck, reparameterization, state-dependent
proposals). When $L(P)$ is close to $1$, the obstruction is spread
out, and a purely local modification is unlikely to help: the regime
is the one where tempering, nonlocal proposals, or a different chain
construction tends to be needed.

The normalization in $r(x)=\KL(P(x,\cdot)\|\pi)/\log(1/\pi(x))$
divides each row's divergence by the self-information of starting at
$x$. Rare states have a large denominator, so $r(x)$ tends to be
small there even when the row itself is far from $\pi$, and the
$\pi$-average $\bar r_\pi$ further down-weights them. This is
intentional but not always what an MCMC user wants: if a low-mass
bottleneck is operationally important, the weighted variant
$L_\nu(P)$ from Remark~\ref{rmk:interpretation} with a non-stationary
$\nu$ is the appropriate object.

\paragraph{Estimating $L(P)$ when $\pi$ is unknown.}
In most MCMC applications, $\pi$ is known only up to a normalizing
constant, and the ratio $\KL(P(x,\cdot)\|\pi)/\log(1/\pi(x))$ does
\emph{not} cancel the normalizer. The statistic, therefore, has to be
plugged in once $\pi$ (or its normalizer) has been estimated, by
standard Monte Carlo or annealing methods; we make no robustness claim
about the resulting plug-in estimator.

In a purely empirical setting one may use the occupation measure
$\widehat\pi$ from a pilot trajectory $X_0,\dots,X_T$ and the
empirical row counts to form
$\widehat r(x)=\widehat{\KL}(P(x,\cdot)\|\widehat\pi)/\log(1/\widehat\pi(x))$.
This is consistent as $T\to\infty$ for any state visited infinitely
often, but unstable for states with $\widehat\pi(x)$ at the resolution
of $1/T$, so the maximum defining $\widehat M$ should be restricted
to states visited at least a chosen number of times. A concentration
analysis of $\widehat L(P)$ around $L(P)$ is left to future work.

\subsection*{Limitations and outlook}

The localization ratio $L(P)$ is a \emph{diagnostic} invariant, not a
universal control parameter for $\etaKL(P)/M$. The construction in
Appendix~\ref{app:counterexample} demonstrates that $L(P)\to 0$ does not by
itself force $\etaKL(P)/M\to 1$: the obstruction arises from the
cardinality, not merely the $\pi$-mass, of high-retention states. This
limitation should be kept in mind when interpreting $L(P)$ in practice.

The spectral lower bound (Theorem~\ref{thm:spectral-lower}) is a
conditional result: it provides a nontrivial converse to
Theorem~\ref{thm:spectral-upper} only when the third-moment correction
$T_3$ is controlled, which holds in the bounded-row regime but may fail
for sparse or near-absorbing chains. It should therefore be regarded as a
complementary tool for well-behaved chains rather than a universal estimate.

Finally, the numerical experiments are based on a limited (though diverse)
test suite; the observed lack of correlation between $L(P)$ and
classical invariants is suggestive but not definitive.

\subsection*{Open questions}

\begin{enumerate}
\item \emph{Characterizing when $\etaKL(P)=M$.}
  Appendix~\ref{app:counterexample} shows that $L(P)\to 0$ alone does
  not force $\etaKL(P)/M\to 1$. A natural refinement asks: for chains in
  which no two high-$r$ states share the same transition row,
  does $L(P)\to 0$ imply $\etaKL(P)/M\to 1$? More generally,
  what structural conditions on the set $\{x:r(x)\approx M\}$ guarantee
  that the point masses are near-optimal?

\item \emph{Spectral lower bound for restricted classes.}
  Does $\bar r_\pi \ge (1/2-o(1))\sum_{i\ge 2}\lambda_i^2/\log n$ hold for
  all lazy random walks on regular graphs? By
  Corollary~\ref{cor:spectral-sandwich}, this reduces to controlling
  $\sum_x\pi(x)T_3(x)$, which is bounded for chains with bounded row entries.

\item \emph{Continuous extension.}
  Can the localization ratio be extended to continuous state spaces, replacing
  the maximum with an essential supremum? For diffusion processes, the analogue
  of $r(x)$ involves the local generator applied to the entropy functional.
  As a heuristic, for a continuous-time chain with generator
  $\mathcal{L}=P-I$ on a finite state space, entropy dissipation is
  governed by the Dirichlet form
  $\mathcal{E}(f,\log f) = \sum_{x,y}\pi(x)P(x,y)(f(y)-f(x))
  (\log f(y)-\log f(x))$. A local retention analog would evaluate this
  dissipation at point-mass test functions, connecting the present framework
  to the entropy--energy landscape studied in PDEs and diffusion theory.
  We leave the precise formulation and the extension to general state
  spaces (where measurability and essential-supremum issues arise) as future
  work.
\end{enumerate}

\section*{Declarations}
\addcontentsline{toc}{section}{Declarations}

\textbf{Funding:} The author declares that no funds, grants, or other
support were received during the preparation of this manuscript.

\noindent \textbf{Conflicts of Interest:} The author has no relevant financial or
non-financial interests to disclose.

\noindent \textbf{Code and Data Availability:} No external datasets were used.
The accompanying reproducibility files are present in the public repository:
\href{https://github.com/SRJ00/Retention-Profile-in-MC}
{Retention Profiles and KL Contraction Bounds in Finite Markov Chains.} 
The repository contains code files, a self-contained audit script, result tables, 
and a requirements file. The script constructs the finite transition matrices described in the text,
computes the row-retention quantities, evaluates point-mass and structured
candidates explicitly, calibrates near-stationary candidates using the local
quadratic bound, and performs a multi-start softmax search. The audit was run
with fixed seed $20260511$ in Python 3.12.13 using NumPy 2.3.5. These numerical
outputs are lower-bound estimates for a nonconvex variational problem unless
an analytic argument is stated separately.

\noindent \textbf{Author Contributions:} Saurav Jadhav is the sole author and is
responsible for all conceptualization, formal analysis, and writing.
\appendix

\section{Discrete-Time Contraction and Exponential Decay}
\label{app:contraction}

\begin{proposition}[One-step contraction]
\label{prop:onestep}
For every probability distribution $\mu$ on $S$,
\[
  \KL(\mu P\|\pi)\le \etaKL(P)\,\KL(\mu\|\pi).
\]
\end{proposition}

\begin{proof}
Immediate from the definition of $\etaKL(P)$.
\end{proof}

\begin{proposition}[Multi-step contraction]
\label{prop:multi-step}
For every integer $t\ge 1$ and every distribution $\mu$ on $S$,
\[
  \KL(\mu P^t\|\pi)\le \etaKL(P)^t\,\KL(\mu\|\pi).
\]
\end{proposition}

\begin{proof}
Apply Proposition~\ref{prop:onestep} iteratively:
\[
  \KL(\mu P^t\|\pi)\le \etaKL(P)\KL(\mu P^{t-1}\|\pi)
  \le \etaKL(P)^2\KL(\mu P^{t-2}\|\pi)
  \le \cdots \le \etaKL(P)^t\KL(\mu\|\pi).
\]
\end{proof}

\begin{corollary}[Exponential form]
\label{cor:exp}
If $\etaKL(P)>0$, then for every $t\ge 1$,
\[
  \KL(\mu P^t\|\pi)\le e^{-\rateKL(P)t}\,\KL(\mu\|\pi).
\]
\end{corollary}

\begin{proof}
$\etaKL(P)^t=e^{t\log\etaKL(P)}=e^{-\rateKL(P)t}$.
\end{proof}

\section{Taylor Remainder Estimate}
\label{app:taylor}

\begin{lemma}
\label{lem:taylor-remainder}
Let $\varphi(u)=(1+u)\log(1+u)$ for $u>-1$. Define
$R(u) = \varphi(u) - u - u^2/2$. Then
$R(u) \ge -u^3/6$ for all $u > -1$; in particular,
$R(u)\ge -|u|^3/6$.
\end{lemma}

\begin{proof}
Define $h(u) := R(u) + u^3/6$. Compute:
\begin{align*}
  h'(u) &= \log(1+u) - u + \tfrac{u^2}{2}, \\
  h''(u) &= \frac{u^2}{1+u} \ge 0 \quad\text{for all }u>-1.
\end{align*}
So $h'$ is non-decreasing; since $h'(0)=0$, we have $h'(u)\le 0$ for $u<0$
and $h'(u)\ge 0$ for $u>0$, meaning $h$ achieves its global minimum at
$u=0$, where $h(0)=0$. Hence $h(u)\ge 0$ for all $u>-1$, i.e.,
$R(u)\ge -u^3/6$. Since $-u^3/6\ge -|u|^3/6$, the second form follows.
\end{proof}

\section{Entropy Contraction Computation for Point-Mass Test Functions}
\label{app:mlsi-detail}

We verify $\Ent_\pi(Pf_x) = \KL(P(x,\cdot)\|\pi)$ for
$f_x(y) = \delta_{xy}/\pi(x)$ when $P$ is reversible.

The operator $P$ acts on functions via 
\[
  (Pf)(y) = \sum_z P(y,z)f(z).
\]
For $f_x(z) = \delta_{zx}/\pi(x)$, we have $(Pf_x)(y) = P(y,x)/\pi(x)$. By reversibility, $\pi(y)P(y,x) = \pi(x)P(x,y)$, so
\[
  (Pf_x)(y) = P(x,y)/\pi(y).
\]
Because $\sum_y (Pf_x)(y)\pi(y) = \sum_y P(x,y) = 1$,
\begin{align*}
  \Ent_\pi(Pf_x)
  &= \sum_y (Pf_x)(y)\log(Pf_x)(y)\,\pi(y) \\
  &= \sum_y P(x,y)\log\frac{P(x,y)}{\pi(y)} \\
  &= \KL(P(x,\cdot)\|\pi).
\end{align*}

\section{Formal Construction: Limits of \texorpdfstring{$L(P)$}{L(P)} as a Predictor}
\label{app:counterexample}

We formalize the construction showing that $L(P_n)\to 0$ does not imply
$\etaKL(P_n)/M_n\to 1$. The mechanism relies on the fact that $L(P)$ controls the total
$\pi$-mass of high-$r$ states, but not their cardinality. When many low-mass states share
identical transition rows, the mutual information penalty $I_\mu(X;Y)$ is forced to zero,
allowing the entropy inflation to drive the contraction ratio.

\begin{theorem}
\label{thm:counterexample}
For any $\varepsilon\in(0,1/2)$ there exists a family of reversible Markov
chains $P_{\delta,N}$ (parameterized by $\delta\in(0,1)$ and integer
$N\ge 1$) on a finite state space, with
$M_{\delta,N}:=M(P_{\delta,N})$, such that
\[
  \lim_{N\to\infty}\;\lim_{\delta\to 0}\;L(P_{\delta,N})\;\le\;2\varepsilon,
  \qquad
  \lim_{N\to\infty}\;\lim_{\delta\to 0}\;\frac{\etaKL(P_{\delta,N})}{M_{\delta,N}}\;\ge\;2.
\]
Consequently, $L(P)\to 0$ does not universally imply
$\etaKL(P)/M\to 1$.
\end{theorem}

\begin{remark}[Nature of the limit]
The construction is a three-parameter family in $(\varepsilon,\delta,N)$
with $N=|B|$, and the conclusion is read as an iterated limit
$\lim_{N\to\infty}\lim_{\delta\to 0}$. For each $\varepsilon\in(0,1/2)$
the inner limit $\delta\to 0$ produces a chain on $A\cup B$ whose
block-to-block transitions vanish, so the limiting kernel is reducible;
the outer limit $N\to\infty$ then sends the per-state mass
$\pi(b)=(1-\varepsilon)/N$ to zero, which forces the retention $r(b)$
on the spectator block to zero (cf.\ the calculation in the proof).
What the result rules out is the implication
$L(P_n)\to 0\Rightarrow\etaKL(P_n)/M_n\to 1$ as a structural statement,
not at fixed connectivity. The obstruction it isolates is the cardinality
of high-$r$ states (here $|A|=k\approx 1/\varepsilon$), independent of
their total $\pi$-mass.
\end{remark}

\begin{proof}
Fix $\varepsilon\in(0,1/2)$, set $k=\lceil 1/\varepsilon\rceil$, and let
$\delta\in(0,1)$, $N\ge 1$ be parameters. Take a state space
$S=A\cup B$ with $|A|=k$ and $|B|=N$, and define
$\pi(a)=\varepsilon/k$ for $a\in A$ and $\pi(b)=(1-\varepsilon)/N$ for
$b\in B$.
Define the transition rows for states $a\in A$ to be identical:
\[
  P(a,x)=
  \begin{cases}
    (1-\delta)/k & x\in A,\\[2pt]
    \delta/N & x\in B,
  \end{cases}
\]
and for $b\in B$:
\[
  P(b,x)=
  \begin{cases}
    \dfrac{\varepsilon\delta}{k(1-\varepsilon)} & x\in A,\\[6pt]
    \dfrac{1-\varepsilon-\varepsilon\delta}{N(1-\varepsilon)} & x\in B.
  \end{cases}
\]
Rows sum to $1$ and $\pi P=\pi$ by inspection. Reversibility holds: for
$a\in A$, $b\in B$,
\[
  \pi(a)P(a,b)=\frac{\varepsilon\delta}{kN}=\pi(b)P(b,a),
\]
and symmetry within $A$ and within $B$ is immediate.

Because all states in $A$ share the same transition row
$q_\delta:=P(a,\cdot)$, they have a common retention value $r(a)$.
Both $q_\delta$ and $\pi$ are uniform within $A$ and within $B$, so the
binary partition $\{A,A^c\}$ captures the full divergence:
\[
  \KL(q_\delta\|\pi) = (1-\delta)\log\frac{1-\delta}{\varepsilon}
  +\delta\log\frac{\delta}{1-\varepsilon}
  \xrightarrow{\delta\to 0}\log(1/\varepsilon).
\]
Thus, for $a \in A$, the retention profile satisfies
\[
  r(a) = \frac{\KL(q_\delta \| \pi)}{\log(k/\varepsilon)} \xrightarrow{\delta \to 0} \frac{\log(1/\varepsilon)}{\log(k/\varepsilon)}.
\]
For $b\in B$, as $\delta\to 0$, $P(b,\cdot)$ converges to the uniform
distribution on $B$ (probability $1/N$). Direct computation gives
\[
  \KL(P(b,\cdot)\|\pi)\xrightarrow{\delta\to 0}
  \sum_{x\in B}\frac{1}{N}\log\frac{1/N}{(1-\varepsilon)/N}
  =\log\!\left(\frac{1}{1-\varepsilon}\right),
\]
so
\[
  r(b)\xrightarrow{\delta\to 0}
  \frac{\log(1/(1-\varepsilon))}{\log(N/(1-\varepsilon))}
  \xrightarrow{N\to\infty}0.
\]
Set
\[
  m_A:=\frac{\log(1/\varepsilon)}{\log(k/\varepsilon)},
  \qquad
  s_N:=\frac{\log(1/(1-\varepsilon))}{\log(N/(1-\varepsilon))}.
\]
Then $m_A>0$ depends only on $\varepsilon$, while $s_N\to0$ as
$N\to\infty$. Hence, for all sufficiently large $N$, the
$\delta\to0$ limit of $M_{\delta,N}$ is $m_A$, and
\[
  \lim_{N\to\infty}\lim_{\delta\to0}L(P_{\delta,N})
  =
  \lim_{N\to\infty}
  \frac{\varepsilon\,m_A+(1-\varepsilon)s_N}{m_A}
  =
  \varepsilon
  \le 2\varepsilon .
\]

Now, consider the uniform initial distribution over the bottleneck: $\mu = \mathrm{Unif}(A)$.
We have $\KL(\mu \| \pi) = \log(1/\varepsilon)$ and $H(\mu) = \log k$.
Because all rows in $A$ are identical, the marginal distribution of one step is exactly $q_\delta$, which means that $\KL(P(a, \cdot) \| \mu P) = 0$ for all $a \in A$. Thus, the mutual information penalty is $I_\mu(X;Y) = 0$.
By Theorem~\ref{thm:convexity-gap},
\[
  \KL(\mu P \| \pi) = \sum_{a \in A} \mu(a) \KL(P(a, \cdot) \| \pi) - I_\mu(X;Y) = \KL(q_\delta \| \pi).
\]
The contraction ratio for this specific $\mu$ is
\[
  \frac{\KL(\mu P \| \pi)}{\KL(\mu \| \pi)} = \frac{\KL(q_\delta \| \pi)}{\log(1/\varepsilon)} \xrightarrow{\delta \to 0} 1.
\]
Since $\etaKL(P)$ is the supremum over all distributions, $\etaKL(P) \ge \KL(\mu P \| \pi) / \KL(\mu \| \pi)$.
Dividing by $M$, we obtain the asymptotic lower bound as follows:
\[
  \lim_{\delta \to 0} \frac{\etaKL(P)}{M} \ge \frac{1}{\lim_{\delta \to 0} M} = \frac{\log(k/\varepsilon)}{\log(1/\varepsilon)} = 1 + \frac{\log k}{\log(1/\varepsilon)}.
\]
Because we chose $k=\lceil 1/\varepsilon\rceil\ge 1/\varepsilon$, we have
$\log k\ge\log(1/\varepsilon)$, hence
\[
  \lim_{\delta\to 0}\frac{\etaKL(P)}{M}
  \;\ge\; 1+\frac{\log k}{\log(1/\varepsilon)}
  \;\ge\; 2.
\]
The ceiling slack $\log\lceil 1/\varepsilon\rceil-\log(1/\varepsilon)\in[0,\varepsilon]$
goes in the favourable direction; in particular, the bound is independent
of $\varepsilon$. Sending $\varepsilon\to 0$ then drives $L(P)\to 0$
while keeping $\etaKL(P)/M\ge 2$, completing the proof.
\end{proof}

\section{Tensorization}
\label{sec:tensor}

\begin{proposition}[Product chain retention]
\label{prop:tensor}
Let $P_1$ and $P_2$ be Markov kernels on finite state spaces $S_1, S_2$ with
stationary distributions $\pi_1, \pi_2$ respectively. Consider the product
chain $P = P_1\otimes P_2$ on $S_1\times S_2$ with $\pi=\pi_1\otimes\pi_2$.
Then, for each $(x_1,x_2)\in S_1\times S_2$:
\[
  r_P(x_1,x_2)
  = \frac{\log(1/\pi_1(x_1))\cdot r_1(x_1)
  + \log(1/\pi_2(x_2))\cdot r_2(x_2)}
  {\log(1/\pi_1(x_1)) + \log(1/\pi_2(x_2))},
\]
where $r_1, r_2$ are the retention profiles of $P_1, P_2$ respectively.
\end{proposition}

\begin{proof}
By independence,
\[
  \KL(P_1(x_1,\cdot)\otimes P_2(x_2,\cdot)\|\pi_1\otimes\pi_2)
  = \KL(P_1(x_1,\cdot)\|\pi_1) + \KL(P_2(x_2,\cdot)\|\pi_2),
\]
and
$\log(1/\pi(x_1,x_2)) = \log(1/\pi_1(x_1)) + \log(1/\pi_2(x_2))$.
The retention of the product chain at $(x_1,x_2)$ is the ratio of these,
which is a weighted average of $r_1(x_1)$ and $r_2(x_2)$ with weights
$\log(1/\pi_j(x_j))$.
\end{proof}

\begin{corollary}[Tensorization bounds]
\label{cor:tensor-bounds}
Under the hypotheses of Proposition~\ref{prop:tensor}:
\begin{enumerate}
  \item $M_P \le \max(M_1, M_2)$.

  \item If both $\pi_1, \pi_2$ are uniform on $n_1, n_2$ states respectively:
  \[
    \bar{r}_P = \frac{\bar{r}_1 \log n_1 + \bar{r}_2 \log n_2}{\log(n_1 n_2)}.
  \]

  \item If $P_1 = P_2$ with uniform $\pi$, then $\bar{r}_P = \bar{r}_1$ and $L(P) \ge L(P_1)$.
\end{enumerate}
\end{corollary}

\begin{proof}
(1) $r_P(x_1,x_2)$ is a convex combination of $r_1(x_1)\le M_1$ and
$r_2(x_2)\le M_2$, so $r_P(x_1,x_2)\le\max(M_1,M_2)$.

(2) When $\log(1/\pi_j(x_j))=\log n_j$:
$r_P(x_1,x_2) = (r_1(x_1)\log n_1 + r_2(x_2)\log n_2)/\log(n_1n_2)$.
Averaging over $\pi_1\otimes\pi_2$ gives $\bar r_P = (\bar r_1\log n_1 + \bar r_2\log n_2)/\log(n_1n_2)$.

(3) If $n_1=n_2=n$ and $P_1=P_2=Q$, then $\bar r_P = \bar r_Q$ and
$M_P\le M_Q$, so $L(P) = \bar r_P/M_P\ge \bar r_Q/M_Q = L(Q)$.
\end{proof}

\begin{remark}[Significance]
The bound $M_P\le\max(M_1,M_2)$ mirrors the exact tensorization property of
the global contraction coefficient,
$\etaKL(P_1\otimes P_2) = \max(\etaKL(P_1),\etaKL(P_2))$
(see Polyanskiy--Wu~\cite{PolyanskiyWu2017}), demonstrating that worst-case
local retention scales consistently with global contraction. This is also
consistent with the general principle that entropy behaves well under products
(see Bobkov--Tetali~\cite{BobkovTetali2006}
and Cesi~\cite{Cesi2001} for related factorization results).
In the identical uniform two-fold product covered by Corollary~\ref{cor:tensor-bounds}(3),
the inequality $L(P)\ge L(P_1)$ means that this tensorization step weakly
delocalizes obstructions rather than concentrating them further.
\end{remark}


\end{document}